\documentclass[twoside,11pt,reqno]{amsart}
\usepackage{amsmath,amssymb,amscd,mathrsfs,epic,wasysym,latexsym,tikz,mathrsfs,cite,hyperref}
\usepackage{pb-diagram}

\makeatletter

\numberwithin{equation}{section}

\newtheorem{Proposition}[equation]{Proposition}
\newtheorem{Lemma}[equation]{Lemma}
\newtheorem{Theorem}[equation]{Theorem}
\newtheorem{Corollary}[equation]{Corollary}
\theoremstyle{definition}  

\newtheorem{Remark}[equation]{Remark}

\newtheorem{Example}[equation]{Example}

\newtheorem{Conjecture}[equation]{Conjecture}

\let\<\langle
\let\>\rangle

\newcommand\Comment[2][\relax]{\space\par\medskip\noindent%
   \fbox{\begin{minipage}{\textwidth}\textbf{Comment\ifx\relax#1\else---#1\fi}\newline%
        #2\end{minipage}}\medskip
}

\def\bi{\text{\boldmath$i$}}
\def\bj{\text{\boldmath$j$}}

\def\b1{\text{\boldmath$1$}}

\newcommand{\Hom}{\operatorname{Hom}}

\newcommand{\im}{\operatorname{im}}

\newcommand{\soc}{\operatorname{soc}}
\newcommand{\head}{\operatorname{head}}

\newcommand{\infl}{\operatorname{infl}}

\newcommand{\Z}{\mathbb{Z}}

\def\eps{{\varepsilon}}
\def\phi{{\varphi}}

\newcommand{\ga}{\gamma}

\newcommand{\la}{\lambda}
\newcommand{\La}{\Lambda}
\newcommand{\al}{\alpha}
\newcommand{\be}{\beta}

\def\Si{\mathfrak{S}}
\newcommand{\si}{\sigma}

\newcommand{\de}{\delta}
\newcommand{\De}{\Delta}

\newcommand{\Ind}{{\mathrm {Ind}}}
\newcommand{\Coind}{{\mathrm {Coind}}}

\newcommand{\pr}{{\mathrm {pr}}}

\newcommand{\Res}{{\mathrm {Res}}}

\newcommand{\Q}{{\mathbb Q}}

\newcommand{\A}{{\mathscr A}}

\renewcommand{\mod}{\bmod \,}

\def\h{{\mathfrak h}}
\def\g{{\mathfrak g}}
\def\n{{\mathfrak n}}

\def\Par{{\mathscr P}}
\def\ula{{\underline{\lambda}}}
\def\umu{{\underline{\mu}}}
\def\unu{{\underline{\nu}}}
\def\uka{{\underline{\kappa}}}

\def\f{{\mathbf{f}}}
\def\b{\mathfrak{b}}
\def\k{\Bbbk}

\def\height{{\operatorname{ht}}}

\def\wt{{\operatorname{wt}}}

\def\re{{\mathrm{re}}}
\def\im{{\mathrm{im}}}

\def\into{{\hookrightarrow}}
\def\Mod#1{#1\!\operatorname{-Mod}}
\def\mod#1{#1\!\operatorname{-mod}}

\def\Mde{M}

\def\lan{\langle}
\def\ran{\rangle}

\def\Stand{\Delta}

\def\CH{{\operatorname{ch}_q\:}}
\def\DIM{{\operatorname{dim}_q\:}}

\def\words{{\langle I\rangle}}
\def\shift{{\tt sh}}

\def\Seq{{\tt Se}}
\def\Car{{\tt C}}
\def\cc{{\tt c}}
\def\Alc{{\mathcal C}}
\def\barinv{\mathtt{b}}

{\catcode`\|=\active
  \gdef\set#1{\mathinner{\lbrace\,{\mathcode`\|"8000%
  \let|\midvert #1}\,\rbrace}}
}
\def\midvert{\egroup\mid\bgroup}

\colorlet{darkgreen}{green!50!black}
\tikzset{dots/.style={very thick,loosely dotted},
         greendot/.style={fill,circle,color=darkgreen,inner sep=1.5pt,outer sep=0}
}
\def\greendot(#1,#2){\node[greendot] at(#1,#2){}}

\newenvironment{braid}{
  \begin{tikzpicture}[baseline=6mm,blue,line width=1pt, scale=0.4,
                      draw/.append style={rounded corners},
                      every node/.append style={font=\fontsize{5}{5}\selectfont}]%
  }{\end{tikzpicture}
}

\def\Grid(#1,#2){
  \draw[very thin,gray,step=2mm] (0,0)grid(#1,#2);
  \draw[very thin,darkgreen,step=10mm] (0,0)grid(#1,#2);
}

\newcommand\Tableau[2][\relax]{
  \begin{tikzpicture}[scale=0.5,draw/.append style={thick,black}]
    \ifx\relax#1\relax%
    \else 
      \foreach\box in {#1} { \filldraw[blue!30]\box+(-.5,-.5)rectangle++(.5,.5); }
    \fi
    \newcount\row\newcount\col
    \row=0
    \foreach \Row in {#2} {
       \col=1
       \foreach\k in \Row {
          \draw(\the\col,\the\row)+(-.5,-.5)rectangle++(.5,.5);
          \draw(\the\col,\the\row)node{\k};
          \global\advance\col by 1
       }
       \global\advance\row by -1
    }
  \end{tikzpicture}
}

\newcommand\YoungDiagram[2][\relax]{
  \begin{tikzpicture}[scale=0.5,draw/.append style={thick,black}]
    \ifx\relax#1\relax%
    \else 
    \foreach\box in {#1} {
      \filldraw[blue!30]\box rectangle ++(1,1);
    }
    \fi
    \newcount\row
    \row=0
    \foreach \col in {#2} {
       \draw(1,\the\row)grid ++(\col,1);
       \global\advance\row by -1
    }
  \end{tikzpicture}
}

\begin{document}


\title[Cuspidal sustems for affine KLR algebras]{{\bf Cuspidal systems for affine Khovanov-Lauda-Rouquier algebras}}

\author{\sc Alexander S. Kleshchev}
\address{Department of Mathematics\\ University of Oregon\\
Eugene\\ OR~97403, USA}
\email{klesh@uoregon.edu}



\thanks{
Research supported in part by the NSF grant no. DMS-1161094 and the Humboldt Foundation.}

\begin{abstract}
A cuspidal system for an affine Khovanov-Lauda-Rouquier algerba $R_\al$ yields  a theory of  standard modules. This allows us to classify the irreducible modules over $R_\al$ up to the so-called imaginary modules. We make a conjecture on reductions modulo $p$ of irreducible $R_\al$-modules, which generalizes James Conjecture. We also describe minuscule imaginary modules, laying the groundwork for future study of imaginary Schur-Weyl duality. We introduce colored imaginary tensor spaces and reduce a classification of imaginary modules to one color. 
We study the characters of cuspidal modules.  
We show that under the Khovanov-Lauda-Rouquier categorification, cuspidal modules correspond to dual root vectors. 
\end{abstract}

\maketitle

\section{Introduction}
Khovanov-Lauda-Rouquier (KLR) algebras were defined in \cite{KL1,KL2,R}. Their representation theory is of interest for the theory of canonical bases, modular representation theory, cluster theory, knot theory, and other areas of mathematics. Let $F$ be an arbitrary ground field. The KLR algebra $R_\al=R_\al(\Car,F)$ is a graded unital associative $F$-algebra depending on a Lie type $\Car$ and an element $\al$ of the non-negative part $Q_+$ of the corresponding root lattice. 

A natural approach to representation theory of $R_\al$ is provided by a theory of standard modules. For KLR algebras of finite Lie type such a theory was first described in \cite{KRbz}, see also \cite{HMM,BKOP, McN}. Key features of this theory are as follows. There is a natural induction functor $\Ind_{\al,\be}$, which associates to an $R_\al$-module $M$ and an $R_\be$-module $N$ the $R_{\al+\be}$-module 
$$M\circ N:=\Ind_{\al,\be} M\boxtimes N$$ for $\al,\be\in Q_+$. We refer to this operation as the {\em induction product}. The functor $\Ind_{\al,\be}$ has an obvious right adjoint $\Res_{\al,\be}$. 

To every positive root $\be\in\Phi_+$ of the corresponding root system $\Phi$, one associates a {\em cuspidal module} $L_\be$. We point out a remarkable property of cuspidal modules which turns out to be key for building the theory of standard modules: the induction product powers $L_\be^{\circ n}$ are irreducible for all $n>0$, see \cite[Lemma 6.6]{KRbz}. 
We make a special choice of a total order on $\Phi_+$, and let $\be_1>\dots>\be_N$ be the positive roots taken in this order. A {\em root partition} of $\al\in Q_+$ is a tuple $\pi=(m_1,\dots,m_N)$ of nonnegative integers such that $\al=\sum_{n=1}^Nm_n\be_n$. The set of root partitions of $\al$ is denoted by $\Pi(\al)$. 

Given $\pi=(m_1,\dots,m_N)\in \Pi(\al)$ we define the corresponding standard module $\De(\pi)$ as the induction product  
$$
\De(\pi)=L_{\be_1}^{\circ m_1}\circ\dots\circ L_{\be_N}^{\circ m_N}\langle\shift(\pi)\rangle,
$$ 
where $\langle\shift(\pi)\rangle$ means that grading is shifted by an explicit integer $\shift(\pi)$. Then the head of $\De(\pi)$ is proved to be irreducible, and,  denoting this head by $L(\pi)$, we get a complete irredundant system  
$$\{L(\pi)\mid\pi\in\Pi(\al)\}$$ 
of irreducible $R_\al$-modules. Moreover, the decomposition matrix 
$$([\De(\pi):L(\si)])_{\pi,\si\in\Pi(\al)}$$ is unitriangular if we order its rows and columns according to the natural lexicographic order on root partitions. 

We now comment on the order on $\Phi_+$. In \cite{KRbz}, the so-called Lyndon order is used, cf. \cite{Leclerc}. This is determined by a choice of a total order on the set $I$ of simple roots. Once such a choice has been made, we have a lexicographic order on the set $\words_\al$ of words of weight $\al$. These words play the role of weights in representation theory of $R_\al$. In particular each $R_\al$-module has its highest word, and the highest word of an irreducible module determines the irreducible module uniquely up to an isomorphism. This leads to the natural notion of dominant words, namely the ones which occur as highest words in $R_\al$-modules (called good words in \cite{KRbz}). The dominant  words of cuspidal modules are characterized among all dominant words  by the property that they are Lyndon words. It turns out that the dominant Lyndon words are in one-to-one correspondence with positive roots, and now we can compare positive roots by comparing the corresponding dominant Lyndon words lexicographically. This gives a total order on $\Phi_+$ called a Lyndon order. We point out that the cuspidal modules themselves depend on the choice of a Lyndon order on $\Phi_+$. 

It is well-known that each Lyndon order is convex. However,  there are in general more convex orders on $\Phi_+$ than Lyndon orders.  Recently McNamara \cite{McN} has found a remarkable generalization of the standard module theory which works for any convex order on $\Phi_+$. In this generalization the cuspidal modules are defined via their restriction properties, which seems to be not quite as explicit as the definition via highest words. However, all the other important features of the theory, including the simplicity of induction powers of cuspidal modules, as well as the unitriangularity of decomposition matrices, remain the same. 

In this paper, we begin to extend the results described above from finite to affine root systems. To describe the results in more detail we need some notation. Let the Lie type $\Car$ be of arbitrary {\em untwisted affine type}. In particular, the simple roots are labeled by the elements of $I=\{0,1,\dots,l\}$. 
We have an (affine) root system $\Phi$ and the 
subset $\Phi_+\subset \Phi$ of {\em positive roots}. 
It is known that $\Phi_+=\Phi_+^\re\sqcup \Phi_+^\im
$, where 
$\Phi_+^\re$ are the {\em real roots}, and 
$
\Phi_+^\im=\{n\de\mid n\in\Z_{>0}\}$, for the {\em null-root $\de$},  
are the {\em imaginary roots}.

Following \cite{BKT}, we define a {\em convex preorder} on $\Phi_+$ as a preorder $\preceq$ such that the following three conditions hold for all $\be,\ga\in\Phi_+$:
\begin{eqnarray}
\label{EPO1}
&\be\preceq\ga \ \text{or}\ \ga\preceq \be;
\\
\label{EPO2}
&\text{if $\be\preceq \ga$ and $\be+\ga\in\Phi_+$, then $\be\preceq\be+\ga\preceq\ga$};
\\
&\label{EPO3}
\text{$\be\preceq\ga$ and $\ga\preceq\be$ if and only if $\be$ and $\ga$ are proportional}.
\end{eqnarray}
Convex preorders are known to exist. It follows from (\ref{EPO3}) that $\be\preceq\ga$ and $\ga\preceq\be$ happens for $\be\neq\ga$ if and only if both $\be$ and $\ga$ are imaginary. Moreover, it is easy to see that the set of real roots splits into two disjoint infinite sets
$$
\Phi^\re_{\succ}:=\{\be\in \Phi_+^\re\mid \be\succ\de\}\ \text{and}\ 
\Phi^\re_{\prec}:=\{\be\in \Phi_+^\re\mid \be\prec\de\}. 
$$
(We write $\be\prec\ga$ if $\be\preceq\ga$ but $\ga\not\preceq\be$). 
In fact, one can label the real roots as 
$$\Phi_+^\re= \{\rho_n\mid n\in\Z_{\neq 0}\}$$ so that 
\begin{equation}\label{EOrderRoots}
\Phi^\re_{\succ}=\{\rho_1\succ\rho_2\succ\rho_3\succ\dots\}\ \text{and}\ 
\Phi^\re_{\prec}=\{\dots\succ\rho_{-3}\succ\rho_{-2}\succ\rho_{-1}\}. 
\end{equation}

Root partitions are defined similarly to the case of finite root systems, except that now we need to take care of imaginary roots. We do this as follows. Let $\al\in Q_+$. Define the set $\Pi(\al)$ of root partitions of $\al$ to be the set of all pairs $(M,\umu)$, where $M=(m_1,m_2,\dots;m_0;\dots,m_{-2},m_{-1})$ is a sequence of nonnegative integers, and $\umu$ is an $l$-multipartition of $m_0$ such that $m_0\de+\sum_{n\neq 0}m_n\rho_n=\al$. There is a natural partial order `$\leq$' on $\Pi(\al)$, which is a version of McNamara's bilexicographic order \cite{McN}, see (\ref{EBilex}). 

A {\em cuspidal system} (for a fixed convex preorder) is the following data:
\begin{enumerate}
\item[{\rm (Cus1)}] An irreducible $R_\rho$-module $L_\rho$ assigned to every $\rho\in \Phi_+^\re$, with the following property: if $\be,\ga\in Q_+$ are non-zero elements such that $\rho=\be+\ga$ and $\Res_{\be,\ga}L_\rho\neq 0$, then $\beta$ is a sum of (positive) roots less than $\rho$ and $\ga$ is a sum of (positive) roots greater than $\rho$. 
\item[{\rm (Cus2)}] An irreducible $R_{n\de}$-module $L(\umu)$ assigned to every $l$-multipartition of $n$ for every $n\in\Z_{\geq 0}$, with the following property: if $\be,\ga\in Q_+\setminus\Phi_+^\im$ are non-zero elements such that $n\de=\be+\ga$ and $\Res_{\be,\ga}L(\umu)\neq 0$, then $\beta$ is a sum of real roots less than $\de$ and $\ga$ is a sum of real roots greater than $\de$. It is required that $L(\ula)\not\cong L(\umu)$ unless $\ula=\umu$. 
\end{enumerate}

We call the irreducible modules $L_\rho$ from (Cus1) {\em cuspidal modules}, and 
the irreducible modules $L(\umu)$ from (Cus2) {\em imaginary modules}. 
It will be proved that cuspidal systems exist for all convex preorders, and cuspidal modules (for a fixed preorder) are determined uniquely up to an  isomorphism. However, it is clearly not the case for imaginary modules: they are defined up to a permutation of multipartitions $\umu$ of $n$. 
We give more comments on this after the Main Theorem. 

Now, given a root partition $(M,\umu)\in\Pi(\al)$ as above, we define the corresponding {\em standard module}
\begin{equation*}
\Stand(M,\umu):=L_{\rho_1}^{\circ m_1} \circ L_{\rho_2}^{\circ m_2}\circ \dots\circ L(\umu)\circ \dots \circ L_{\rho_{-2}}^{\circ m_{-2}}\circ L_{\rho_{-1}}^{m_{-1}} \lan\shift(M,\umu)\ran,
\end{equation*}
where $\shift(M,\mu)$ is an explicit integer defined in (\ref{EShift}). 

\vspace{2mm}
{\bf Main Theorem.} 
{\em 
For any convex preorder there exists a cuspidal system $\{L_\rho\mid \rho\in \Phi_+^\re\}\cup\{L(\ula)\mid \ula\in\Par\}$. Moreover: 
\begin{enumerate}
\item[{\rm (i)}] For every root partition $(M,\umu)$, the standard module  
$
\Stand(M,\umu)
$ has irreducible head; denote this irreducible module $L(M,\umu)$. 

\item[{\rm (ii)}] $\{L(M,\umu)\mid (M,\umu)\in \Pi(\al)\}$ is a complete and irredundant system of irreducible $R_\al$-modules up to isomorphism.

\item[{\rm (iii)}] $L(M,\umu)^\circledast\simeq L(M,\umu)$.  

\item[{\rm (iv)}] $[\Stand(M,\umu):L(M,\umu)]_q=1$, and $[\Stand(M,\umu):L(N,\unu)]_q\neq 0$ implies $(N,\unu)\leq (M,\umu)$. 


\item[{\rm (v)}] $L_\rho^{\circ n}$ is irreducible for every $\rho\in \Phi_+^\re$ and every $n\in\Z_{>0}$. 
\end{enumerate}
}
\vspace{2mm}

This theorem, proved in Section~\ref{SRough}, gives a `rough classification' of irreducible $R_\al$-modules. The main problem is that we did not give a canonical definition of individual imaginary modules $L(\umu)$. We just know that the amount of such modules for $R_{n\de}$ is equal to the number of $l$-multipartitions of $n$, and so we have labeled them by such multipartitions in an arbitrary way. 
In fact, there is a solution to this problem. It turns out that there is a beautiful rich theory of imaginary representations of KLR algebras of affine type, which relies on the so-called imaginary Schur-Weyl duality. This theory in particular allows us to construct an equivalence between an appropriate category of imaginary representations of KLR algebras and the category of representations of the classical Schur algebras. 
We will address these matters in the forthcoming work \cite{Kimag}. 

In Section~\ref{SMinusc}, we make some first steps in the study of imaginary representations and describe explicitly the {\em minuscule} imaginary representations---the ones which correspond to the $l$-multipartitions of $1$. 
We introduce colored imaginary tensor spaces and reduce a classification of imaginary modules to one color. 
Minuscule imaginary representations are also used in Sections~\ref{SS+} and \ref{SS-} to describe explicitly the cuspidal modules corresponding to the roots of the form $n\de\pm \al_i$. In Section~\ref{SCusp} we also explain how the characters of other cuspidal modules can be computed by induction using the idea of minimal pairs which was suggested in \cite{McN}. In Section~\ref{SSCMDRE}, we show that under the Khovanov-Lauda-Rouquier categorification, cuspidal modules correspond to dual root vectors of a dual PBW basis. 

In conclusion, we would like to draw the reader's attention to Conjecture~\ref{ConjJamesKLR}, which asserts that reductions modulo $p$ of irreducible modules over the KLR algebras of affine type remain irreducible under an {\em explicit}\, assumption on the characteristic $p$. In type $A_l^{(1)}$ (for level $1$) this is equivalent to a block version of the James Conjecture. 

Immediately after the first version of this paper has been posted, the paper \cite{TW} has also been released on the arXiv. That paper suggestes a different approach to  standard module theory for affine KLR algebras. 

\subsection*{Acknowledgements} This paper has been influenced by the beautiful ideas of Peter McNamara \cite{McN}, who also drew my attention to the paper \cite{BKT} and suggested a slightly more general version of the main result appearing here after the first version of this paper was released. I am also grateful to Arun Ram and Jon Brundan for many useful conversations. 

\section{Preliminaries}

Throughout the paper, $F$ is a 
field 
of arbitrary  characteristic $p\geq 0$. 
Denote the ring of Laurent polynomials in the indeterminate $q$ by $\A:=\Z[q,q^{-1}]$. We use quantum integers $[n]_q:=(q^n-q^{-n})/(q-q^{-1})\in\A$ for $n\in \Z$, and the quantum factorials $[n]^!_q:=[1]_q[2]_q\dots[n]_q$. We have a bar-involution on $\A$ and on $\Q(q)\supset\A$ with $\barinv q=q^{-1}$. 

\subsection{Lie theoretic notation}\label{SSLTN}
Throughout the paper $\Car=(\cc_{ij})_{i,j\in I}$ is a {\em Cartan matrix} of {\em untwisted affine type}, see \cite[\S 4, Table Aff 1]{Kac}. 
We have $I=\{0,1,\dots,l\}$, where $0$ is the affine vertex. 
Following \cite[\S 1.1]{Kac}, let $(\h,\Pi,\Pi^\vee)$ be a realization of the Cartan matrix $\Car$, so we have simple roots $\{\al_i\mid i\in I\}$, simple coroots $\{\al_i^\vee\mid i\in I\}$, and a bilinear form $(\cdot,\cdot)$ on $\h^*$ such that $\cc_{ij}=2(\al_i,\al_j)/(\al_i,\al_i)$ for all $i,j\in I$. We normalize $(\cdot,\cdot)$ so that $(\al_i,\al_i)=2$ if $\al_i$ is a short simple root. 

The fundamental dominant weights $\{\La_i\mid i\in I\}$ have the property that $\lan\La_i,\al_j^\vee\ran=\de_{i,j}$, where $\lan\cdot,\cdot\ran$ is the natural pairing between $\h^*$ and $\h$. We have the integral weight lattice $P=\oplus_{i\in I}\Z\cdot\La_i$ and the set of  dominant weights $P_+=\sum_{i\in I}\Z_{\geq 0}\cdot\La_i$. 
For $i\in I$ we define
$$
[n]_i:=[n]_{q^{(\al_i, \al_i)/2}},\qquad [n]^!_i:=[1]_i[2]_i\dots[n]_i.
$$
Denote 
$Q_+ := \bigoplus_{i \in I} \Z_{\geq 0} \al_i$. For $\alpha \in Q_+$, we write $\height(\alpha)$ for the sum of its 
coefficients when expanded in terms of the $\al_i$'s.


Let $\g'=\g(\Car')$ be the finite dimensional simple Lie algebra  whose Cartan matrix $\Car'$ corresponds to the subset of vertices $I':=I\setminus\{0\}$. The affine Lie algebra $\g=\g(\Car)$ is then obtained from $\g'$ by a procedure described in \cite[Section 7]{Kac}. We denote by $W$ (resp. $W'$) the corresponding {\em affine Weyl group} (resp. {\em finite Weyl group}). It is a Coxeter group with standard generators $\{r_i\mid i\in I\}$ (resp. $\{r_i\mid i\in I'\}$), see \cite[Proposition~3.13]{Kac}.

Let $\Phi'$ and $\Phi$ be the root systems of $\g'$ and $\g$ respectively. Denote by $\Phi'_+$ and $\Phi_+$ the set of {\em positive}\, roots in $\Phi'$ and $\Phi$, respectively, cf. \cite[\S 1.3]{Kac}. Denote by $\de$ the null-root. Let 
\begin{equation}\label{EDelta}
\de=a_0\al_0+a_1\al_1+\dots+a_l\al_l.
\end{equation}
By \cite[Table Aff 1]{Kac}, we always have 
\begin{equation}\label{Ea_0}
a_0=1.
\end{equation}
We have 
\begin{equation}\label{EHRoot}
\de-\al_0=\theta,
\end{equation} 
where $\theta$ is the highest root in the finite root system $\Phi'$. 
Finally, 
$$\Phi_+=\Phi_+^\im\sqcup \Phi_+^\re,$$ 
where
$$
\Phi_+^\im=\{n\de\mid n\in\Z_{>0}\}
$$
and 
\begin{equation}\label{EUnion}
\Phi_+^\re=\{\be+n\de\mid \be\in  \Phi'_+,\ n\in\Z_{\geq 0}\}\sqcup \{-\be+n\de\mid \be\in  \Phi'_+,\ n\in\Z_{> 0}\}.
\end{equation}

\subsection{Words}
Sequences of elements of $I$ will be called {\em words}. The set of all words is denoted $\words$. If $\bi=i_1\dots i_d$ is a word, we denote $|\bi|:=\al_{i_1}+\dots+\al_{i_d}\in Q_+$. For any $\al\in Q_+$ we denote 
$$\words_\al:=\{\bi\in\words \mid |\bi|=\al\}.$$ 
If $\al$ is of height $d$, then the symmetric group $\Si_d$ with simple permutations $s_1,\dots,s_{d-1}$ acts on $\words_\alpha$ from the left by place permutations.

Let $\bi=i_1\dots i_d$ and $\bj=i_{d+1}\dots i_{d+f}$ be two elements of $\words$. 
 Define the {\em quantum shuffle product}: 
 \begin{equation*}
\bi\circ\bj:=\sum q^{-e(\sigma)}i_{\sigma(1)}\dots i_{\sigma(d+f)} \in\A\words,
\end{equation*}
where the sum is over all $\sigma\in S_{d+f}$ such that $\sigma^{-1}(1)<\dots<\sigma^{-1}(d)$ and $\sigma^{-1}(d+1)<\dots<\sigma^{-1}(d+f)$, and 
$
e(\sigma):=\sum_{k\leq d<m,\ \sigma^{-1}(k)>\sigma^{-1}(m)} \cc_{i_{\si(k)}, i_{\sigma(m)}}.
$ 
This defines an $\A$-algebra structure on the $\A$-module $\A\words$, which consists of all finite formal $\A$-linear combinations of elements $\bi\in \words$.

\subsection{KLR algebras} 
Define the polynomials in the variables $u,v$ 
$$\{Q_{ij}(u,v)\in F[u,v]\mid i,j\in I\}$$ 
as follows. For the case where the Cartan matrix $\Car\neq {\tt A}_1^{(1)}$, 
choose signs $\eps_{ij}$ for all $i,j \in I$ with $\cc_{ij}
< 0$  so that $\eps_{ij}\eps_{ji} = -1$.
Then set: 
\begin{equation}\label{EArun}
Q_{ij}(u,v):=
\left\{
\begin{array}{ll}
0 &\hbox{if $i=j$;}\\
1 &\hbox{if $\cc_{ij}=0$;}\\
\eps_{ij}(u^{-\cc_{ij}}-v^{-\cc_{ji}}) &\hbox{if $\cc_{ij}<0$ and $i< j$.}
\end{array}
\right.
\end{equation}
For type $A_1^{(1)}$ we define
\begin{equation}\label{EArun1}
Q_{ij}(u,v):=
\left\{
\begin{array}{ll}
0 &\hbox{if $i=j$;}\\
(u-v)(v-u) &\hbox{if $i\neq j$.}
\end{array}
\right.
\end{equation}

Fix 
$\al\in Q_+$ of height $d$. 
The {\em KLR-algebra} $R_\al$ is an associative graded unital $F$-algebra, given by the generators
\begin{equation}\label{EKLGens}
\{1_\bi\mid \bi\in \words_\al\}\cup\{y_1,\dots,y_{d}\}\cup\{\psi_1, \dots,\psi_{d-1}\}
\end{equation}
and the following relations for all $\bi,\bj\in \words_\al$ and all admissible $r,t$:
\begin{equation}
1_\bi  1_\bj = \de_{\bi,\bj} 1_\bi ,
\quad{\textstyle\sum_{\bi \in \words_\alpha}} 1_\bi  = 1;\label{R1}
\end{equation}
\begin{equation}\label{R2PsiY}
y_r 1_\bi  = 1_\bi  y_r;\quad y_r y_t = y_t y_r;
\end{equation}
\begin{equation}
\psi_r 1_\bi  = 1_{s_r\bi} \psi_r;\label{R2PsiE}
\end{equation}
\begin{equation}
(y_t\psi_r-\psi_r y_{s_r(t)})1_\bi  
= 
\left\{
\begin{array}{ll}
1_\bi  &\hbox{if $i_r=i_{r+1}$ and $t=r+1$,}\\
-1_\bi  &\hbox{if $i_r=i_{r+1}$ and $t=r$,}\\
0 &\hbox{otherwise;}
\end{array}
\right.
\label{R6}
\end{equation}
\begin{equation}
\psi_r^21_\bi  = Q_{i_r,i_{r+1}}(y_r,y_{r+1})1_\bi 
 \label{R4}
\end{equation}
\begin{equation} 
\psi_r \psi_t = \psi_t \psi_r\qquad (|r-t|>1);\label{R3Psi}
\end{equation}
\begin{equation}
\begin{split}
&(\psi_{r+1}\psi_{r} \psi_{r+1}-\psi_{r} \psi_{r+1} \psi_{r}) 1_\bi  
\\=
&\left\{\begin{array}{ll}
\frac{Q_{i_r,i_{r+1}}(y_{r+2},y_{r+1})-Q_{i_r,i_{r+1}}(y_r,y_{r+1})}{y_{r+2}-y_r}1_\bi &\text{if $i_r=i_{r+2}$,}\\
0 &\text{otherwise.}
\end{array}\right.
\end{split}
\label{R7}
\end{equation}
The {\em grading} on $R_\al$ is defined by setting:
$$
\deg(1_\bi )=0,\quad \deg(y_r1_\bi )=(\al_{i_r},\al_{i_r}),\quad\deg(\psi_r 1_\bi )=-(\al_{i_r},\al_{i_{r+1}}).
$$

\vspace{2 mm}

It is pointed out in \cite{KL2} and \cite[\S3.2.4]{R} that up to isomorphism the graded $F$-algebra $R_\al$ depends only on the Cartan matrix and $\al$. 

Fix in addition a dominant weight $\La\in P_+$. The corresponding {\em cyclotomic KLR algebra} $R_\al^\La$ is the quotient of $R_\al$ by the following ideal:
\begin{equation}\label{ECyclot}
J_\al^\La:=(y_1^{\lan\La,\al_{i_1}^\vee\ran}1_\bi \mid \bi=(i_1,\dots,i_d)\in\words_\al). 
\end{equation}

For each element $w\in S_d$ fix a reduced expression $w=s_{r_1}\dots s_{r_m}$ and set 
$$
\psi_w:=\psi_{r_1}\dots \psi_{r_m}.
$$
In general, $\psi_w$ depends on the choice of the reduced expression of $w$. 

\begin{Theorem}\label{TBasis}{\cite[Theorem 2.5]{KL1}}, \cite[Theorem 3.7]{R} 
The elements 
$$ \{\psi_w y_1^{m_1}\dots y_d^{m_d}1_\bi \mid w\in S_d,\ m_1,\dots,m_d\in\Z_{\geq 0}, \ \bi\in \words_\al\}
$$ 
form an $F$-basis of  $R_\al$. 
\end{Theorem}

There exists a homogeneous algebra anti-involution 
\begin{equation}\label{ECircledast}
\tau:R_\al\longrightarrow R_\al,\quad 1_\bi\mapsto 1_\bi,\quad y_r\mapsto y_r,\quad \psi_s\mapsto \psi_s  
\end{equation}
for all $\bi\in \words_\al,\ 1\leq r\leq d$, and $1\leq s<d$. If $M=\bigoplus_{d\in\Z}M_d$ is a finite dimensional 
graded $R_\al$-module, then the {\em graded dual}
$M^\circledast$ is the graded $R_\al$-module such that $(M^\circledast)_n:=\Hom_F(M_{-n},F)$, for all
$n\in\Z$, and the $R_\al$-action is given by $(xf)(m)=f(\tau(x)m)$, for all $f\in M^\circledast, m\in M, x\in
R_\al$.

\subsection{\boldmath Basic representation theory of $R_\al$}\label{SSBasicRep} 
For any ($\Z$-)graded $F$-algebra $H$, we denote by 
$\mod{H}$ 
the abelian subcategory of all
{\em finite dimensional}\, graded $H$-modules, with morphisms being {\em degree-preserving} module homomorphisms, and 
 $[\mod{H}]$ denotes the corresponding Grothendieck group. Then $[\mod{H}]$ is an $\A$-module via 
$
q^m[M]:=[M\langle m\rangle],
$ 
where $M\langle m\rangle$ denotes the module obtained by 
shifting the grading up by $m$, i.e. 
$
M\langle m\rangle_n:=M_{n-m}.
$ 
We denote by $\hom_H(M,N)$ the space of morphism in $\mod{H}$. 
For $n \in \Z$, let
$
\Hom_H(M, N)_n := \hom_H(M \langle n \rangle, N)
$
denote the space of all homomorphisms
that are homogeneous of degree $n$.
Set
$$
\Hom_H(M,N) := \bigoplus_{n \in \Z} \Hom_H(M,N)_n. 
$$
For graded $H$-modules $M$ and $N$ we write $M\simeq N$ to mean that $M$ and $N$ are isomorphic as graded modules and $M\cong N$ to mean that they are isomorphic as $H$-modules after we forget the gradings. 
For a finite dimensional 
graded vector space $V=\oplus_{n\in \Z} V_n$, its {\em graded dimension} is $\DIM \, V:=\sum_{n \in \Z}  (\dim V_n)q^n\in\A$. 
Given $M, L \in \mod{H}$ with $L$ irreducible, 
we write $[M:L]_q$ for the corresponding {\em  graded composition multiplicity},
i.e. 
$
[M:L]_q := \sum_{n \in \Z} a_n q^n,
$ 
where $a_n$ is the multiplicity
of $L\langle n\rangle$ in a graded composition series of $M$.

Going back to the algebras $R_\al$, every irreducible graded $R_\al$-module is finite dimensional \cite[Proposition 2.12]{KL1}, and  there are finitely many irreducible modules in $\mod{R_\al}$ up to isomorphism and grading shift \cite[\S 2.5]{KL1}. A prime field is a splitting field for $R_{\al}$ \cite[Corollary 3.19]{KL1}, so working with irreducible $R_\al$-modules we do not need to assume that $F$ is algebraically closed. 
Finally, for every irreducible module $L$, there is a unique choice of the grading shift so that we have $L^\circledast \simeq L$ \cite[Section 3.2]{KL1}. When speaking of irreducible $R_\al$-modules we often assume by fiat that the shift has been chosen in this way. 

For $\bi\in \words_\al$ and $M\in\mod{R_\al}$, the {\em $\bi$-weight space} of $M$ is
$
M_\bi:=1_\bi M.
$
We have 
$
M=\bigoplus_{\bi\in \words_\al}M_\bi.
$
We say that $\bi$ is a {\em weight of $M$} if $M_\bi\neq 0$.
Note from the relations that 
$
\psi_r M_\bi\subset M_{s_r \bi}.
$
Define the {\em  (graded formal) character} of $M$ as follows: 
\begin{equation*}
\CH M:=\sum_{\bi\in \words_\al}(\DIM M_\bi) \bi \in \A\words_\al.
\end{equation*}
The  character map $\CH: \mod{R_\al}\to \A\words_\al$ factors through to give an {\em injective} $ \A$-linear map 
$
\CH: [\mod{R_\al}]\to  \A\words_\al, 
$ see \cite[Theorem 3.17]{KL1}. 

\subsection{Induction, coinduction, and duality}
Given $\alpha, \beta \in Q_+$, we set $
R_{\alpha,\beta} := R_\alpha \otimes 
R_\beta$.  
Let $M \boxtimes N$ be 
the outer tensor product of the $R_\alpha$-module $M$ and the $R_\beta$-module 
$N$.
There is an injective homogeneous non-unital algebra homomorphism 
$R_{\alpha,\beta}\,\into\, R_{\alpha+\beta},\ 1_\bi \otimes 1_\bj\mapsto 1_{\bi\bj}$,
where $\bi\bj$ is the concatenation of $\bi$ and $\bj$. The image of the identity
element of $R_{\alpha,\beta}$ under this map is
$$
1_{\alpha,\beta}:= \sum_{\bi \in \words_\alpha,\,\bj \in \words_\beta} 1_{\bi\bj}.
$$ 

Let $\Ind_{\alpha,\beta}^{\alpha+\beta}$ and $\Res_{\alpha,\beta}^{\alpha+\beta}$
be the induction and restriction functors: 
\begin{align*}
\Ind_{\alpha,\beta}^{\alpha+\beta} &:= R_{\alpha+\beta} 1_{\alpha,\beta}
\otimes_{R_{\alpha,\beta}} ?:\mod{R_{\alpha,\beta}} \rightarrow \mod{R_{\alpha+\beta}},\\
\Res_{\alpha,\beta}^{\alpha+\beta} &:= 1_{\alpha,\beta} R_{\alpha+\beta}
\otimes_{R_{\alpha+\beta}} ?:\mod{R_{\alpha+\beta}}\rightarrow \mod{R_{\alpha,\beta}}.
\end{align*}
We often omit upper indices and write simply $\Ind_{\alpha,\beta}$ and $\Res_{\alpha,\beta}$. 
These functors have obvious generalizations to $n\geq 2$ factors: 
\begin{align*}
\Ind_{\ga_1,\dots,\ga_n}
:\mod{R_{\ga_1,\dots,\ga_n}} \rightarrow \mod{R_{\ga_1+\dots+\ga_n}},\\
\Res_{\ga_1,\dots,\ga_n}
:\mod{R_{\ga_1+\dots+\ga_n}}\rightarrow \mod{R_{\ga_1,\dots,\ga_n}}.
\end{align*}
The functor $\Ind_{\ga_1,\dots,\ga_n}
$ is left adjoint to $\Res_{\ga_1,\dots,\ga_n}
$. 
If $M_a\in\Mod{R_{\ga_a}}$, for $a=1,\dots,n$, we define 
\begin{equation}\label{ECircProd}
M_1\circ\dots\circ M_n:=\Ind_{\ga_1,\dots,\ga_n}
M_1\boxtimes\dots\boxtimes M_n. 
\end{equation}
In view of \cite[Lemma 2.20]{KL1}, we have
\begin{equation}\label{ECharShuffle}
\CH(M_1\circ\dots\circ M_n)=\CH(M_1)\circ\dots\circ \CH(M_n).
\end{equation}

The functors of induction and restriction have obvious parabolic analogues. Given a family $(\al^a_b)_{1\leq a\leq n,\ 1\leq b\leq m}$ of elements of $Q_+$, set 
$\sum_{a=1}^n\al^{a}_b=:\be_b$ for all $1\leq b\leq m$. Then we have  functors
$$
\Ind_{\al^{1}_1,\dots,\al^{n}_{1}\,;\,\dots\,;\,\al^{1}_{m},\dots,\al^{n}_{m}}^{\,\be_1\,;\,\dots\,;\,\be_m}\qquad \text{and}\qquad \Res_{\al^{1}_1,\dots,\al^{n}_{1}\,;\,\dots\,;\,\al^{1}_{m},\dots,\al^{n}_{m}}^{\,\be_1\,;\,\dots\,;\,\be_m}
$$


The right adjoint to the functor $\Ind_{\ga_1,\dots,\ga_n}
$ is given by the coinduction:
$$
\Coind_{\ga_1,\dots,\ga_n}
:=\Hom_{R_{\ga_1,\dots,\ga_n}}(1_{\ga_1,\dots,\ga_n}R_{\ga_1+\dots+\ga_n},\,?) 
$$
Induction and coinduction are related as follows:

\begin{Lemma} \label{LLV} {\rm \cite[Theorem 2.2]{LV}} 
Let $\underline{\ga}:=(\ga_1,\dots,\ga_n)\in Q_+^n$, and $V_m\in\mod{R_{\ga_m}}$ for $m=1,\dots,n$. 
Denote
$
d(\underline{\ga})=\sum_{1\leq m<k\leq n}(\ga_m,\ga_k).
$ 
Then 
$$
(\Coind_{\ga_n,\dots,\ga_1}
V_n\boxtimes\dots\boxtimes V_1)
\simeq 
\Ind_{\ga_1,\dots,\ga_n}
V_1\boxtimes\dots\boxtimes V_n\lan d(\underline{\ga})\ran.
$$
\end{Lemma}

\begin{Lemma} \label{LDualInd}
Let $\underline{\ga}:=(\ga_1,\dots,\ga_n)\in Q_+^n$, and $V_m\in\mod{R_{\ga_m}}$ for $m=1,\dots,n$. 
Denote
$
d(\underline{\ga})=\sum_{1\leq m<k\leq n}(\ga_m,\ga_k).
$ 
Then 
$$
(V_1\circ\dots\circ V_n)^\circledast\simeq 
(V_n^\circledast\circ\dots\circ V_1^\circledast)\lan d(\underline{\ga})\ran.
$$
\end{Lemma}
\begin{proof}
Follows from Lemma~\ref{LLV} by uniqueness of adjoint functors as in the proof of \cite[Corollary 3.7.4]{Kbook}
\end{proof}

\subsection{Mackey Theorem}\label{SSMackey}
We state a slight generalization of the Mackey Theorem of Khovanov and Lauda \cite[Proposition~2.18]{KL1}. 
Given $x\in \Si_n$ and $\underline{\ga}=(\ga_1,\dots,\ga_n)\in Q_+^n$, we denote 
$$
x\underline{\ga}:=(\ga_{x^{-1}(1)},\dots,\ga_{x^{-1}(n)}). 
$$
Correspondingly, define the integer
$$
s(x,\underline{\ga}):=-\sum_{1\leq m<k\leq n,\ x(m)>x(k)}(\ga_m,\ga_k).
$$

Writing $R_{\underline{\ga}}$ for $R_{\ga_1,\dots,\ga_n}$, there is an obvious natural algebra isomorphism 
$$
\phi^x:R_{x\underline{\ga}}\to R_{\underline{\ga}}
$$
permuting the components. Composing with this isomorphism, we get a functor
$$
\mod{R_{\underline{\ga}}}\to \mod{R_{x\underline{\ga}}},\  M\mapsto {}^{\phi^x}M.
$$
Making an additional shift, we get a functor 
\begin{equation}\label{ETwist}
\mod{R_{\underline{\ga}}}\to \mod{R_{x\underline{\ga}}},\  M\mapsto {}^xM:={}^{\phi^x}M\langle s(x,\underline{\ga})\rangle.
\end{equation}

\begin{Theorem} \label{TMackeyKL}
Let $\underline{\ga}=(\ga_1,\dots,\ga_n)\in Q_+^n$ and $\underline{\be}=(\be_1,\dots,\be_m)\in Q_+^m$ with $\ga_1+\dots+\ga_n=\be_1+\dots+\be_m=:\al$. Then for any $M\in\mod{R_{\underline{\ga}}}$ we have that $\Res_{\underline{\be}}\,\Ind_{\underline{\ga}} M$ has  filtration with factors of the form
$$
\Ind_{\al^{1}_1,\dots,\al^{n}_{1}\,;\,\dots\,;\,\al^{1}_{m},\dots,\al^{n}_{m}}^{\,\be_1\,;\,\dots\,;\,\be_m}
{}^{x(\underline{\al})}\big(\Res_{\al^{1}_1,\dots,\al^{1}_{m}\,;\,\dots\,;\,\al^{n}_{1},\dots,\al^{n}_{m}}^{\,\ga_1\,;\,\dots\,;\,\ga_n}
\,M \big)
$$
with $\underline{\al}=(\al^a_b)_{1\leq a\leq n,\ 1\leq b\leq m}$ running over all tuples of elements of $Q_+$ such that  $\sum_{b=1}^m\al^{a}_b=\ga_a$ for all $1\leq a\leq n$ and $\sum_{a=1}^n\al^{a}_b=\be_b$ for all $1\leq b\leq m$, and $x(\underline{\al})$ is the permutation of $mn$ which maps 
$$
(\al^{1}_1,\dots,\al^{1}_{m};\al^{2}_1,\dots,\al^{2}_{m};\dots;\al^{n}_{1},\dots,\al^{n}_{m})
$$
to
$$
(\al^{1}_1,\dots,\al^{n}_{1};\al^{1}_2,\dots,\al^{n}_{2};\dots;\al^{1}_{m},\dots,\al^{n}_{m}).
$$
\end{Theorem}
\begin{proof}
For $m=n=2$ this follows from \cite[Proposition~2.18]{KL1}. The general case can be proved by the same argument or deduced from the case $m=n=2$ by induction. 
\end{proof}

\subsection{Crystal operators}
The theory of crystal operators has been developed in \cite{KL1}, \cite{LV} and \cite{KK} following ideas of Grojnowski \cite{Gr}, see also \cite{Kbook}. We review necessary facts for reader's convenience. 

Let $\al\in Q_+$ and $i\in I$. It is known that $R_{n\al_i}$ is a nil-Hecke algebra with unique (up to a degree shift) irreducible module, which we denote by $L(i^n)$. Moreover, $\DIM L(i^n)=[n]^!_i$. 
We have functors 
\begin{align*}
&e_i: \mod{R_\al}\to\mod{R_{\al-\al_i}},\ M\mapsto \Res^{R_{\al-\al_i,\al_i}}_{R_{\al-\al_i}}\circ \Res_{\al-\al_i,\al_i}M,
\\
&f_i: \mod{R_\al}\to\mod{R_{\al+\al_i}},\ M\mapsto \Ind_{\al,\al_i}M\boxtimes L(i).
\end{align*}
If $L\in\mod{R_\al}$ is irreducible, we define
$$
\tilde f_i L:=\head (f_i L),\quad \tilde e_i L:=\soc (e_i L).
$$
A fundamental fact is that $\tilde f_i L$ is again irreducible and $\tilde e_i L$ is irreducible or zero. 
We refer to $\tilde e_i$ and $\tilde f_i$ as the crystal operators. These are operators on $B\cup\{0\}$, where $B$ is the set of isomorphism classes of  irreducible $R_\al$-modules for all $\al\in Q_+$.  
Define
$
\wt:B\to P,\ [L]\mapsto -\al
%
$
if $L\in\mod{R_\al}$.

\begin{Theorem} \label{TLV} {\rm \cite{LV}} 
The set $B$ with the operators $\tilde e_i,\tilde f_i$ and the function $\wt$ is the crystal graph of the negative part $U_q(\n_-)$ of the quantized enveloping algebra of $\g$.  
\end{Theorem}

For any $M\in\mod{R_\al}$, we define
$$
\eps_i(M):=\max\{k\geq 0\mid e_i^k(M)\neq 0\}.
$$
Then $\eps_i(M)$ is also the length of the longest `$i$-tail' of weights of $M$, i.e. the maximum of $k\geq 0$ such that $j_{d-k+1}=\dots=j_d=i$ for some weight $\bj=(j_1,\dots,j_d)$ of $M$. Define also
$$
\eps_i^*(M):=\max\{k\geq 0\mid j_1=\dots=j_k=i\ \text{for a weight $\bj=(j_1,\dots,j_d)$ of $M$}\}
$$
to be the length of the longest `$i$-head' of weights of $M$. 

\begin{Proposition} \label{PCryst1} {\rm \cite{LV,KL1}} 
Let $L$ be an irreducible $R_\al$-module, $i\in I$, and $\eps=\eps_i(L)$. 
\begin{enumerate}
\item[{\rm (i)}] $\tilde e_i\tilde f_iL\cong L$ and if $\tilde e_i L\neq 0$ then $\tilde f_i\tilde e_iL\cong L$; 
\item[{\rm (ii)}] $\eps=\max\{k\geq 0\mid \tilde e_i^k(L)\neq 0\}$;
\item[{\rm (iii)}] $\Res_{\al-\eps\al_i,\eps\al_i}L\cong \tilde e_i^\eps L\boxtimes L(i^\eps)$.
\end{enumerate}
\end{Proposition}

Recall from (\ref{ECyclot}) the cyclotomic ideal $J_\al^\La$. We have an obvious  functor of inflation $\infl^\La:\mod{R_\al^\La}\to \mod{R_\al}$ and its left adjoint 
$$
\pr^\La:\mod{R_\al}\to \mod{R_\al^\La},\ M\mapsto M/J_\al^\La M.
$$

\begin{Lemma} \label{LPr} {\rm \cite[Proposition 2.4]{LV}} 
Let $L$ be an irreducible $R_\al$-module. Then $\pr^\La L\neq 0$ if and only if $\eps_i^*(L)\leq \lan\La,\al_i^\vee\ran$ for all $i\in I$. 
\end{Lemma}

\subsection{Extremal words and multiplicity one results} \label{SSCOES}
Let $i\in I$. Consider the map  
$
\theta_i^*:\words\to\words
$ such that for $\bj=(j_1,\dots,j_d)\in\words$, we have
\begin{equation}\label{ETheta*}
\theta_i^*(\bj)=
\left\{
\begin{array}{ll}
(j_1,\dots,j_{d-1}) &\hbox{if $j_d=i$;}\\
0 &\hbox{otherwise.}
\end{array}
\right.
\end{equation}
We extend $\theta_i^*$ by linearity to a map $\theta_i^*:\A\words\to\A\words$. 

Let $x$ be an element of $\A\words$. Define 
$$\eps_i(x):=\max\{k\geq 0\mid (\theta_i^*)^{k}(x)\neq 0\}.$$ 
A word $i_1^{a_1}\dots i_b^{a_b}\in\words$, with  $a_1,\dots,a_b\in\Z_{\geq 0}$, is called {\em extremal} for $x$ if $a_b=\eps_{i_b}(x)$, $a_{b-1}=\eps_{i_{b-1}}((\theta_{i_b}^*)^{a_b}(x))$, \dots , $a_1=\eps_{i_1}((\theta_{i_2}^*)^{a_2}\dots (\theta_{i_b}^*)^{a_b}(x))$. A weight $i_1^{a_1}\dots i_b^{a_b}\in\words_\al$ is called {\em extremal} for $M\in\mod{R_\al}$ if it is an extremal word for $\CH M\in\A\words$, in other words, if  
$a_b=\eps_{i_b}(M)$, $a_{b-1}=\eps_{i_{b-1}}(\tilde e_{i_b}^{a_b}M)$, \dots , $a_1=\eps_{i_1}(\tilde e_{i_2}^{a_2}\dots\tilde e_{i_b}^{a_b}M)$.   

The following useful result, which is a version of \cite[Corollary 2.17]{BKdurham}, describes the multiplicities of extremal weight spaces in irreducible modules. We denote by $1_F$ the trivial module $F$ over the trivial algebra $R_0\simeq F$.

\begin{Lemma} \label{LExtrMult}
Let $L$ be an irreducible $R_\al$-module, and $\bi=i_1^{a_1}\dots i_b^{a_b}\in\words_\al$ be an extremal weight for $L$. Then 
$\DIM L_\bi=[a_1]^!_{i_1}\dots [a_b]^!_{i_b}$, and 
$$L\cong \tilde f_{i_b}^{a_b} \tilde f_{i_{b-1}}^{a_{b-1}}\dots\tilde f_{i_1}^{a_1}1_F.$$  Moreover, $\bi$ is not an extremal weight for any irreducible module $L'\not\cong L$. 
\end{Lemma}
\begin{proof}
Follows easily from Proposition~\ref{PCryst1}, cf. \cite[Theorem 2.16]{BKdurham}. 
\end{proof}

\begin{Corollary} \label{CExtrNew}
Let $M\in\mod{R_\al}$, and $\bi=i_1^{a_1}\dots i_b^{a_b}\in\words_\al$ be an extremal weight for $M$. Then we can write $\DIM M_\bi=m[a_1]^!_{i_1}\dots [a_b]^!_{i_b}$ for some $m\in\A$. Moreover, if $L\cong \tilde f_{i_b}^{a_b} \tilde f_{i_{b-1}}^{a_{b-1}}\dots\tilde f_{i_1}^{a_1}1_F$ and $L^\circledast\simeq L$, then we have $[M:L]_q=m$. 
\end{Corollary}
\begin{proof}
Apply Lemma~\ref{LExtrMult}, cf. \cite[Corollary 2.17]{BKdurham}. 
\end{proof}

Now we establish some useful `multiplicity-one results'. The first one shows that in every irreducible module there is a weight space with a one dimensional graded component:

\begin{Lemma} \label{LMultOneWeight}
Let $L$ be an irreducible $R_\al$-module, and $\bi=i_1^{a_1}\dots i_b^{a_b}\in\words_\al$ be an extremal weight for $L$. Set 
$N:=-b+\sum_{m=1}^b a_m(\al_{i_m},\al_{i_m})/2.$ 
Then $\dim 1_\bi L_N=\dim 1_\bi L_{-N}=1$. 
\end{Lemma}
\begin{proof}
This follows immediately from the equality $\DIM 1_\bi L=[a_1]^!_{i_1}\dots [a_b]^!_{i_b}$, which comes from Lemma~\ref{LExtrMult}. 
\end{proof}

The following result shows that any induction product of irreducible modules always has a multiplicity one composition factor.

\begin{Proposition} \label{PProdIrrMult1}
Suppose that $n\in\Z_{>0}$ and for $r=1,\dots,n$, we have $\al^{(r)}\in Q_+$, an irreducible $R_{\al^{(r)}}$-module $L^{(r)}$, and  $\bi^{(r)}:=i_1^{a^{(r)}_1}\dots i_k^{a^{(r)}_k}\in\words_{\al^{(r)}}$ is an extremal weight for $L^{(r)}$. Denote $a_m:=\sum_{r=1}^na^{(r)}_m$ for all $1\leq m\leq k$. 
Then 
$
\bj:=i_1^{a_1}\dots i_k^{a_k}
$
is an extremal weight for $L^{(1)}\circ \dots \circ L^{(n)}$, 
and the graded multiplicity of the $\circledast$-self-dual irreducible module 
$$N\cong\tilde f_{i_k}^{a_k} \tilde f_{i_{k-1}}^{a_{k-1}}\dots\tilde f_{i_1}^{a_1}1_F$$ 
in $L^{(1)}\circ \dots \circ L^{(n)}$ is $q^{m}$,
where
$$m:=-\textstyle\sum_{1\leq t<u\leq n}\left(\sum_{1\leq r< s\leq k}a_r^{(u)}a_s^{(t)}(\al_{i_r},\al_{i_s})+\frac{1}{2}\sum_{r=1}^k a_r^{(t)}a_r^{(u)}(\al_{i_r},\al_{i_r})\right). 
$$
In particular, the ungraded multiplicity of $N$ in $L^{(1)}\circ \dots \circ L^{(n)}$ is one.  
\end{Proposition}
\begin{proof}
By Lemma~\ref{LExtrMult}, the multiplicity of $\bi^{(r)}$ in $\CH L^{(r)}$ is  $[a^{(r)}_1]_{i_1}^{!}\dots [a^{(r)}_k]_{i_k}^!$. By (\ref{ECharShuffle}), we have 
$$\CH(L^{(1)}\circ \dots \circ L^{(n)})=\CH(L^{(1)})\circ \dots \circ \CH(L^{(n)}).$$ 
It is easy to see that the weight $\bj$  
is an extremal weight for $L^{(1)}\circ \dots \circ L^{(n)}$, and that $\bj$ can be obtained only from the shuffle product $\bi^{(1)}\circ \dots\circ \bi^{(n)}$. An elementary  computation shows that $\bj$ appears in $\bi^{(1)}\circ \dots\circ \bi^{(n)}$ with multiplicity  $q^{m}[a_1]_{i_1}^!\dots [a_k]_{i_k}^!$. 
Now apply Corollary~\ref{CExtrNew}. 
\end{proof}


\begin{Corollary} \label{CPowerIrr}
Let $L$ be an irreducible $R_\al$-module and $n\in\Z_{>0}$. Then there is an irreducible $R_{n\al}$-module $N$ which appears in $L^{\circ n}$ with graded multiplicity $q^{-(\al,\al)n(n-1)/4}$. In particular, the ungraded multiplicity of $N$ is one.  
\end{Corollary}
\begin{proof}
Apply Proposition~\ref{PProdIrrMult1} with $L^{(1)}=\dots= L^{(n)}=L$. 
\end{proof}

\subsection{Khovanov-Lauda-Rouquier categorification}\label{SSKLRCat}
We recall the Khovanov-Lauda-Rouquier categorification of the quantized enveloping algebra $\f$ obtained in \cite{KL1,KL2,R}. We follow the presentation of \cite{KRbz, BKM}. Let $\f_\A\subset \f$ be the $\A$-form of the Lusztig's quantum group $\f$ corresponding to the Cartan matrix $\Car$. This $\A$-algebra is generated by the divided powers $\theta_i^{(n)}=\theta_i^n /[n]_i^!$ of the standard generators. 
The algebra $\f_\A$ has a $Q_+$-grading $\f_\A=\oplus_{\al\in Q_+}(\f_\A)_\al$ determined by the condition that each $\theta_i$ is in degree $\al_i$. 

There is a bilinear form $(\cdot,\cdot)$ on $\f$ defined in \cite[$\S$1.2.5, $\S$33.1.2]{Lubook}. Let $\f_\A^*= \left\{y \in \f\:\big|\:(x,y)\in\A\text{ for all }x \in \f_\A\right\}$. Let $(\theta_i^*)^{(n)}$ be the map dual to the map $\mathbf{f}_\A\to \mathbf{f}_\A,\ x\mapsto x\theta_i^{(n)}$. Finally, there is a coproduct $r$  on $\f$ such that $\f$ is a twisted unital and counital bialgebra. Moreover, for all $x,y,z\in \f$ we have 
\begin{equation}\label{EDefPropForm}
(xy,z)=(x\otimes y,r(z)).
\end{equation}

The field $\Q(q)$ possesses a unique automorphism called the
{\em bar involution} such that $\overline{q} = q^{-1}$.
With respect to this involution, 
let $\barinv:\f \rightarrow \f$
be the anti-linear algebra automorphism
such that $\barinv(\theta_i) = \theta_i$ for all $i \in I$.
Also let 
$\barinv^*:\f \rightarrow \f$
be the adjoint anti-linear
map to $\barinv$ 
with respect to Lusztig's form, so 
$(x, \barinv^*(y)) = \overline{(\barinv(x), y)}$
for all $x, y \in \f$.
The maps $\barinv$ and $\barinv^*$ preserve $\f_\A$ and $\f_\A^*$, respectively. 

Let $[\mod{R}] = \bigoplus_{\alpha \in Q_+} [\mod{R_\alpha}]$
denote the Grothendieck ring, which is an $\A$-algebra via induction product and  
 $q^n [V] = [V\lan n\ran]$. Similarly the functors of restriction define a coproduct $r$  on $[\mod{R}]$. This product and coproduct make $[\mod{R}]$ into a twisted unital and counital bialgebra \cite[Proposition 3.2]{KL1}.

 It \cite{KL1,KL2} an explicit $\A$-bialgebra isomorphisms
$
\ga^*:[\mod{R}] \stackrel{\sim}{\rightarrow}
\f_\A^*
$
is constructed; in fact \cite{KL1} establishes a dual isomorphism, see \cite[Theorem 4.4]{KRbz} for details on this. Moreover, $\ga^*([V^\circledast])=\barinv^*(\ga^*([V]))$, and we have a commutative triangle
\begin{equation}\label{ETriangle}
\begin{pb-diagram}
\node{}\node{\A\words}
\node{} \\
\node{[\mod{R}]} \arrow[2]{e,t}{\ga^*}
\arrow{ne,t}{\CH}
\node{}\node{\mathbf{f}_\A^*}
\arrow{nw,t}{\iota}
\end{pb-diagram},
\end{equation}
where the map $\iota$ is defined as follows: 
$$\iota(x)=\sum_{\bi=(i_1,\dots,i_d)\in\words}(x,\theta_{i_1}\dots\theta_{i_d})\bi\qquad(x\in \mathbf{f}_\A^*).$$ 

\begin{Lemma} \label{LExtrDCB}
Let $v^*$ be a dual canonical basis element of $\f$, and $\bi=i_1^{a_1}\dots i_k^{a_k}$ be an extremal word of $\iota(v^*)$ in the sence of Section~\ref{SSCOES}. Then $\bi$ appears in $\iota(v^*)$  with coefficient $[a_1]_{i_1}^!\dots[a_k]_{i_k}^!$. 
\end{Lemma}
\begin{proof}
Apply induction on $a_1+\dots+a_k$. The induction base is  $a_1+\dots+a_k=0$, in which case $v^*=1\in\f^*_\A$ and $\iota(1)$ is the empty word. 
Recall the map $\theta_i^*:\A\words\to\A\words$ from (\ref{ETheta*}). For all $x\in\f^*_\A$ we have 
$
\iota((\theta_i^*)^{(n)}(x))=(\theta_i^*)^{(n)}(\iota(x))
$, where in the right hand side $(\theta_i^*)^{(n)}=(\theta_i^*)^{n}/[n]_i^!$. 
By \cite[Proposition 5.3.1]{KaG}, 
$(\theta_{i_k}^*)^{(a_{i_k})} (v^*)$
is again a dual canonical basis element, and by induction, the word $i_1^{a_1}\dots i_{k-1}^{a_{k-1}}$ appears in $\iota((\theta_{i_k}^*)^{(a_{i_k})} (v^*))$ with coefficient $[a_1]_{i_1}^!\dots[a_{k-1}]_{i_{k-1}}^!$. The result follows. 
\end{proof}



\section{Cuspidal systems and standard modules}
\subsection{\boldmath Convex preorders on $\Phi_+$}\label{SSSCO}
Recall the notion of a convex preorder on $\Phi_+$ from (\ref{EPO1})--(\ref{EPO3}). Convex preorders exist, see e.g. \cite[Example 2.11(ii)]{BKT}.

\begin{Lemma} \label{LBKT}{\rm \cite{BKT}} 
For any positive root $\be$, the convex cones spanned by $\Phi_+(\be):=\{\ga\in \Phi_+\mid \ga\succeq \be\}$ and  $\Phi_+\setminus \Phi(\be)$ intersect only at the origin. 
\end{Lemma}
\begin{proof}
The set $\{\ga\in \Phi_+\mid \ga\succeq \be\}$ is a terminal section for the preorder $\preceq$ in the sense of \cite[Section 2.4]{BKT}. 
 By \cite[Lemma 2.9]{BKT}, this set is biconvex, which is equivalent to the statement  about the cones by \cite[Remark 2.3]{BKT}. 
\end{proof}

Lemma~\ref{LBKT} immediately implies the following properties: 
\begin{enumerate}
\item[{\rm (Con1)}] Let $\rho\in\Phi_+^\re$, $m\in\Z_{>0}$, and $m\rho=\sum_{a=1}^b \ga_a$ for some positive roots $\ga_a$. Assume that either  $\ga_a\preceq \rho$ for all $a=1,\dots,b$ or $\ga_a\succeq \rho$ for all $a=1,\dots,b$. Then  $b=m$ and $\ga_a=\rho$ for all $a=1,\dots,b$. 
\item[{\rm (Con2)}] Let $\be,\kappa$ be two positive roots, not both imaginary. If $\be+\kappa=\sum_{a=1}^b \ga_a$ for some positive roots $\ga_a\preceq \be$, then $\be\succeq \kappa$. 
\item[{\rm (Con3)}] Let $\rho\in\Phi_+^\im$, and $\rho=\sum_{a=1}^b \ga_a$ for some positive roots $\ga_a$. If either $\ga_a\preceq \rho$ for all $a=1,\dots,b$ or $\ga_a\succeq \rho$ for all $a=1,\dots,b$, then all $\ga_a$ are imaginary. 
\end{enumerate}

Indeed, for (Con1), we may assume that all $\ga_a\prec \rho$, and apply the lemma with $\be=\rho$. For (Con2), taking into account (Con1), we may assume that all $\ga_a\prec \be$, and apply the lemma. For (Con3), we may assume that all $\ga_a$ are real and apply the lemma with $\be=\rho$. 

The Main Theorem from the introduction will be proved for an arbitrary convex preorder, but later results which rely on the theory of imaginary representations, beginning from Section~\ref{SMinusc}, require an additional assumption. Recall from (\ref{EUnion}) that 
$$
\Phi_+^\re=\{\be+n\de\mid \be\in  \Phi'_+,\ n\in\Z_{\geq 0}\}\sqcup \{-\be+n\de\mid \be\in  \Phi'_+,\ n\in\Z_{> 0}\}.
$$
A convex preorder $\preceq$ will be called {\em balanced} if 
\begin{equation}\label{EBalanced}
\Phi^\re_{\succ}=\{\be+n\de\mid \be\in  \Phi'_+,\ n\in\Z_{\geq 0}\}.
\end{equation}
Then of course we also have $\Phi^\re_{\prec}=\{-\be+n\de\mid \be\in  \Phi'_+,\ n\in\Z_{> 0}\}.$ Balanced convex preorders exist, see for example \cite{BCP}. For  reader's convenience we conclude this section with a sketch of a construction of  balanced preorders (it will not be used in the paper). 
Let $V$ be the $\mathbb R$-span of $\Phi'$. The affine group $W$ acts on $V$ with affine transformations, see \cite[Chapter 4]{Hum} . 
This action induces a simply transitive action of $W$ on the set of alcoves, which are the connected components of the complement of the affine hyperplanes 
$$H_{\al,n}=\{v\in V\mid (v,\al)=m\}\qquad(\al\in\Phi'_+,m\in \Z).$$ 
Let 
$\Alc=\{v\in V\mid 0<(v,\al)<1\ \text{for all $\al\in\Phi'_+$}\}$ 
be the fundamental alcove. 
A (possibly infinite) sequence  of alcoves 
$
\dots,\Alc_2,\Alc_1,\Alc_0=\Alc,\Alc_{-1},\Alc_{-2},\dots
$
is called an {\em alcove path} if for each $n$ there is a common wall for  $\Alc_n$ and $\Alc_{n+1}$ so that $\Alc_{n+1}$ is a reflection of $\Alc_n$ in this wall. Let $\be_n\in\Phi'_+$ and $m_n$ be defined from
$$
H_{\be_n,m_n}:=
\left\{
\begin{array}{ll}
\text{the common wall between $\Alc_{n-1}$ and $\Alc_n$} &\hbox{if $n>0$}\\
\text{the common wall between $\Alc_{n+1}$ and $\Alc_n$} &\hbox{if $n<0$}
\end{array}
\right. 
$$
An alcove path as above is called {\em simple} 
if it crosses each affine hyperplane exactly once. An alcove path is called {\em balanced} if $m_n\geq 0$ for all $n>0$ and $m_n<0$ for all $n<0$. For a  balanced simple path  
set
$$
\rho_n:=
\left\{
\begin{array}{ll}
\be_n +m_n\de &\hbox{if $n>0$}\\
-\be_n+m_n\de &\hbox{if $n<0$}.
\end{array}
\right.
$$
Now we define a preorder on $\Phi_+$ such that  
$
\rho_1\succ\rho_2\succ\dots \succ n\de\succeq m\de\succ \dots\succ \rho_{-2}\succ\rho_{-1}
$ for all $n,m\in\Z_{>0}.$ 
The well-known geometric interpretation of the reduced expressions in terms of alcove paths \cite[Section 4.4, 4.5]{Hum} easily implies that this preorder is convex, and it is balanced by definition.

\subsection{Root partitions}\label{SSRP}
Recall that $I'=\{1,\dots,l\}$. We will consider the set $\Par$ of $l$-multipartitions $\ula=(\la^{(i)})_{i\in I'}$, where each $\la^{(i)}=(\la^{(i)}_1,\la^{(i)}_2,\dots)$ is a usual partition. 
For all $i\in I'$, we denote $|\la^{(i)}|:=\la^{(i)}_1+\la^{(i)}_2+\dots$,  and set $|\ula|:=\sum_{i\in I'}|\la^{(i)}|$. For $m\in \Z_{\geq 0}$, 
denote $\Par_m:=\{\ula\in\Par\mid |\ula|=m\}$. 


We work with a fixed convex preorder $\preceq$ on $\Phi_+$. Recall the notation (\ref{EOrderRoots}). We will consider finitary sequences of non-negative integers
of the form
$$
M=(m_1,m_2,\dots;\, m_0;\, \dots,m_{-2},m_{-1}). 
$$
The set of all such sequences is denoted by $\Seq$. The left lexicographic order on $\Seq$ is denoted $\leq_l$ and the right lexicographic order on $\Seq$ is denoted $\leq_r$. We will use the following {\em bilexicographic} partial order on $\Seq$: 
$$
M\leq N\qquad\text{if and only if}\qquad M\leq_l N \ \text{and}\ M\geq_r N.
$$

A {\em root partition} is a pair $(M,\umu)$ with $M\in \Seq$ and $\umu\in\Par_{m_0}$. For a root partition $(M,\umu)$ we define
$$
M_n:=
\left\{
\begin{array}{ll}
m_n\rho_n &\hbox{if $n\neq 0$}\\
m_0\de &\hbox{if $n=0$}
\end{array}
\right.,
$$
and set 
\begin{equation}\label{EM}
|M|=(M_1,M_2,\dots;\, M_0;\, \dots,M_{-2},M_{-1}).
\end{equation}
This is a finitary sequence of elements of $Q_+$. 
If $\sum_{n\in \Z} M_n=\al$ we say that $(M,\umu)$ is a root partition of $\al$. In that case we have a parabolic subalgebra $R_{|M|}\subseteq R_\al$. Denote by $\Pi(\al)$ the set of all root partitions of $\al$. 
We will use the following partial order on $\Pi(\al)$:
\begin{equation}\label{EBilex}
(M,\umu)\leq (N,\unu)\ \text{if and only if}\  M\leq N\ \text{and if $M=N$ then}\ \umu=\unu.
\end{equation}

The positive subalgebra $\n_+\subset \g$ has a basis consisting of {\em root vectors} 
$$\{E_\rho,\ E_{n\de,i}\mid \rho\in \Phi_+^\re,\ n\in\Z_{>0},\ i\in I'\}.$$ 
For $i\in I'$, assign to a partition $\mu^{(i)}=(\mu^{(i)}_1,\mu^{(i)}_2,\dots)$ a PBW monomial $E_{\mu^{(i)}}:=E_{\mu^{(i)}_1\de,i} E_{\mu^{(i)}_2\de,i}\dots$. 
Now, to a root partition $(M,\umu)$, we assign a PBW monomial 
$$
E_{M,\umu}:=E_{\rho_1}^{m_1}\, E_{\rho_2}^{m_2}\ \dots\  E_{\mu^{(1)}}E_{\mu^{(2)}}\ \dots\ E_{\mu^{(l)}}\   \dots\  E_{\rho_{-2}}^{m_{-2}}\, E_{\rho_{-1}}^{m_{-1}}.
$$
Then $\{E_{M,\umu}\mid(M,\umu)\in\Pi(\al)\}$ is a basis of 
 the weight space $U(\n_+)_\al$. In particular, $|\Pi(\al)|=\dim U(\n_+)_\al$ is the {\em Kostant partition function} of $\al$. In view of the isomorphism $\ga^*$ from (\ref{ETriangle}), we conclude:

\begin{Lemma} \label{LAmount}
The number of irreducible $R_\al$-modules (up to isomorphism) is $|\Pi(\al)|$.  
\end{Lemma}

Given a root partition $(M,\umu)$ and $a\in\Z$, denote by $(M,\umu)'_a$ the root partition obtained from $(M,\umu)$ by `annihilating' its $a$th component; to be more precise,  $(M,\umu)'_a=(M',\umu')$, where 
\begin{equation}\label{EM'}
m_b'=
\left\{
\begin{array}{ll}
0 &\hbox{if $b=a$}\\
m_b &\hbox{if $b\neq a$}
\end{array}
\right.
\qquad\text{and}\qquad
\umu'=
\left\{
\begin{array}{ll}
\emptyset &\hbox{if $a=0$}\\
\umu &\hbox{otherwise.}
\end{array}
\right.
\end{equation}

Finally, sometimes we use a slightly different notation for the root partitions. For example, if $(M,\umu)$ is such that $m_1=2,m_2=1, m_{-3}=1,$ and all other $m_a$ with $a\neq 0$ are zero, then we write $(M,\umu)=(\rho_1,\rho_1,\rho_2,\umu,\rho_{-3})$.

\subsection{Standard modules}
We continue to work with a fixed convex preorder $\preceq$ on $\Phi_+$ and use the notation (\ref{EOrderRoots}).  
Recall from the introduction the definition of the corresponding cuspidal system. It consists of certain cuspidal modules $L_\rho$ for $\rho\in\Phi_+^\re$  and  imaginary modules $L(\umu)$ for $\umu\in\Par$ satisfying the properties (Cus1) and  (Cus2). 
For every $\al\in Q_+$ and $(M,\umu)\in\Pi(\al)$, we define an integer
\begin{equation}\label{EShift}
\shift(M,\umu):=\sum_{n\neq 0} (\rho_n,\rho_n)m_n(m_n-1)/4,
\end{equation}
we define the $R_{|M|}$-module
\begin{equation}
L_{M,\umu}:=L_{\rho_1}^{\circ m_1} \boxtimes L_{\rho_2}^{\circ m_2}\boxtimes \dots\boxtimes L(\umu)\boxtimes \dots \boxtimes L_{\rho_{-2}}^{\circ m_{-2}}\boxtimes L_{\rho_{-1}}^{\circ m_{-1}} \lan\shift(M,\umu)\ran, 
\end{equation}
and we define the {\em standard module}
\begin{equation}\label{EStand}
\Stand(M,\umu):=L_{\rho_1}^{\circ m_1} \circ L_{\rho_2}^{\circ m_2}\circ \dots\circ L(\umu)\circ \dots \circ L_{\rho_{-2}}^{\circ m_{-2}}\circ L_{\rho_{-1}}^{\circ m_{-1}} \lan\shift(M,\umu)\ran.
\end{equation}
Note that $\Stand(M,\umu)=\Ind_{|M|} L_{M,\umu}\in\mod{R_\al}$. 

\begin{Lemma} \label{LPowerCuspSelfDual}
Let $\rho\in\Phi_+^\re$, $L_\rho$ be the corresponding cuspidal module, and $n\in \Z_{>0}$. Then $$(L_\rho^{\circ n})^\circledast\simeq L_\rho^{\circ n}\lan(\rho, \rho)n(n-1)/2\ran.$$
In particular, the module $L_\rho^{\circ n}\lan (\rho,\rho)n(n-1)/4\ran$ is $\circledast$-self-dual. 
\end{Lemma}
\begin{proof}
Recall that our standard choice of shifts of irreducible modules is so that $L_\rho^{\circledast}\simeq L_\rho$. Now the result follows from Lemma~\ref{LDualInd}. 
\end{proof}

\begin{Lemma} \label{LLMMuSeldDual}
We have $L_{M,\umu}^\circledast\simeq L_{M,\umu}$
\end{Lemma}
\begin{proof}
Follows from Lemma~\ref{LPowerCuspSelfDual}.
\end{proof}

\subsection{Restrictions of standard modules}
The proof of the following proposition is similar to \cite[Lemma 3.3]{McN}.

\begin{Proposition}\label{P1}
Let $(M,\umu),(N,\unu)\in\Pi(\al)$. Then:
\begin{enumerate}
\item[{\rm (i)}] $\Res_{|N|} \Stand(N,\unu)\simeq L_{N,\unu}$.
\item[{\rm (ii)}] $\Res_{|M|} \Stand(N,\unu)\neq 0$ implies $M\leq N$. 
\end{enumerate}
\end{Proposition}
\begin{proof}
Let $\Res_{|M|}^{\al} \Stand(N,\unu)\neq 0$. It suffices to prove that $M\geq_l N$ or $M\leq_r N$ implies that $M=N$ and  $\Res_{|M|}^{\al} \Stand(N,\unu)\cong L_{N,\unu}$. We may assume that $M\geq_l N$, the case $M\leq _r N$ being similar. 
We apply induction on $\height(\al)$ and consider three cases. 

Case 1: $m_a>0$ for some $a>0$. 
Pick the minimal such $a$, and let $(M',\umu')=(M,\umu)'_a$ and $(N',\unu')=(N,\unu)'_a$, see (\ref{EM'}). 
By the Mackey Theorem~\ref{TMackeyKL}, $\Res_{|M|}^{\al} \Stand(N,\unu)$ has  filtration with factors of the form 
$$
\Ind^{m_a\rho_a;|M'|}_{\kappa_1,\dots,\kappa_c;\underline{\ga}}V,
$$
where $m_a\rho_a=\kappa_1+\dots+\kappa_c$, with $\kappa_1,\dots,\kappa_c\in Q_+\setminus\{0\}$, and $\underline{\ga}$ is a refinement of $|M'|$. Moreover, the module $V$ is obtained by twisting and degree shifting as in (\ref{ETwist}) of a module obtained by restriction of 
$$
L_{\rho_1}^{\boxtimes n_1}\boxtimes L_{\rho_2}^{\boxtimes n_2}\boxtimes\dots \boxtimes L(\unu)\boxtimes\dots \boxtimes L_{\rho_{-2}}^{\boxtimes n_{-2}}\boxtimes L_{\rho_{-1}}^{\boxtimes n_{-1}}
$$
to a parabolic which has $\kappa_1,\dots,\kappa_c$ in the beginnings of the corresponding blocks. In particular, if $V\neq 0$, then for each $b=1,\dots,c$ we have that $\Res^{\rho_k}_{\kappa_b,\rho_k-\kappa_b}L_{\rho_k}\neq 0$ for some $k=k(b)$ with $n_k\neq 0$ or $\Res^{n_0\de}_{\kappa_b,n_0\de-\kappa_b}L(\unu)\neq 0$. 

If $\Res^{\rho_k}_{\kappa_b,\rho_k-\kappa_b}L_{\rho_k}\neq 0$, then by (Cus1), $\kappa_b$ is a sum of roots $\preceq \rho_k$. Moreover, since $M\geq_l N$ and $n_k\neq0$, we have that $\rho_k\preceq\rho_a$. Thus $\kappa_b$ is a sum of roots $\preceq \rho_a$. On the other hand, if $\Res^{n_0\de}_{\kappa_b,n_0\de-\kappa_b}L(\unu)\neq 0$, then by (Cus2), either $\kappa_b$ is an imaginary root or it is a sum of real roots less than $n_0\de$. In either case we conclude again that $\kappa_b$ is a sum of roots $\preceq \rho_a$. Using (Con1), we can now conclude that $c=m_a$, and $\kappa_b=\rho_a=\rho_{k(b)}$ for all $b=1,\dots,c$. Hence $n_a\geq m_a$. Since  $M\geq _l N$, we conclude that $n_a=m_a$, and 
$$
\Res_{|M|}^{\al} \Stand(N,\unu)\cong L_{\rho_a}^{\circ m_a}\boxtimes \Res_{|M'|}^{\al-m_a\rho_a} \Stand(N',\unu'). 
$$
Now, since $\height(\al-m_a\rho_a)<\height(\al)$, we can apply the inductive hypothesis. 

Case 2: $m_b=0$ for all $b>0$, but $m_0\neq 0$. 
Since $N\leq_l M$, we also have that $n_b=0$ for all $b>0$. 
Let $(M',\umu')=(M,\umu)'_a$, $(N',\unu')=(N,\unu)'_a$. By the Mackey Theorem~\ref{TMackeyKL}, $\Res_{|M|}^{\al} \Stand(N,\unu)$ has filtration with factors of the form 
$$
\Ind^{m_0\de;|M'|}_{\kappa_1,\dots,\kappa_c;\underline{\ga}}V,
$$
where $m_0\de=\kappa_1+\dots+\kappa_c$, with $\kappa_1,\dots,\kappa_c\in Q_+\setminus\{0\}$, and $\underline{\ga}$ is a refinement of $|M'|$. Moreover, the module $V$ is obtained by twisting and degree shifting of a module obtained by parabolic restriction of the module 
$
L(\unu)\boxtimes\dots \boxtimes L_{\rho_{-2}}^{\boxtimes n_{-2}}\boxtimes L_{\rho_{-1}}^{\boxtimes n_{-1}}
$
to a parabolic which has $\kappa_1,\dots,\kappa_c$ in the beginnings of the corresponding blocks. In particular, if $V\neq 0$, then either 

(1) $\Res^{n_0\de}_{\kappa_1,n_0\de-\kappa_1}L(\unu)\neq 0$ and for  $b=2,\dots,c$, there is $k=k(b)<0$ such that $\Res^{\rho_k}_{\kappa_b,\rho_k-\kappa_b}L_{\rho_k}\neq 0$,  or 

(2) for $b=1,\dots,c$ there is $k=k(b)<0$ such that $\Res^{\rho_k}_{\kappa_b,\rho_k-\kappa_b}L_{\rho_k}\neq 0.$

\noindent
By (Cus1) and (Con3), only (1) is possible, and in that case, using also (Cus2), we must have $c=1$ and $\kappa_1=m_0\de$. Since $M\geq _l N$, we conclude that $n_0=m_0$, and 
$$
\Res_{|M|}^{\al} \Stand(N,\unu)\cong L(\unu)\boxtimes \Res_{|M'|}^{\al-m_0\de} \Stand(N',\unu). 
$$
Now, since $\height(\al-m_0\de)<\height(\al)$, we can apply the inductive hypothesis. 

Case 3: $m_b=0$ for all $b\geq 0$. This case is similar to Case 1. 
\end{proof}

\section{Rough classification of irreducible modules}\label{SRough}
We continue to work with a fixed convex preorder $\preceq$ on $\Phi_+$ and use the notation (\ref{EOrderRoots}). 
In this section we prove the Main Theorem from the introduction. 

\subsection{Statement and the structure of the proof}
We will prove the following result, which contains slightly more information than the Main Theorem:

\begin{Theorem} \label{THeadIrr} 
For a given convex preorder, there exists a corresponding cuspidal system $\{L_\rho\mid \rho\in \Phi_+^\re\}\cup\{L(\ula)\mid \ula\in\Par\}$. Moreover: 
\begin{enumerate}
\item[{\rm (i)}] For every root partition $(M,\umu)$, the standard module  
$
\Stand(M,\umu)
$ has an irreducible head; denote this irreducible module $L(M,\umu)$. 

\item[{\rm (ii)}] $\{L(M,\umu)\mid (M,\umu)\in \Pi(\al)\}$ is a complete and irredundant system of irreducible $R_\al$-modules up to isomorphism.

\item[{\rm (iii)}] $L(M,\umu)^\circledast\simeq L(M,\umu)$.  

\item[{\rm (iv)}] $[\Stand(M,\umu):L(M,\umu)]_q=1$, and $[\Stand(M,\umu):L(N,\unu)]_q\neq 0$ implies $(N,\unu)\leq (M,\umu)$. 

\item[{\rm (v)}] $\Res_{|M|}L(M,\umu)\simeq L_{M,\umu}$ and $\Res_{|N|}L(M,\umu)\neq 0$ implies $N\leq M$.  

\item[{\rm (vi)}] $L_\rho^{\circ n}$ is irreducible for all $\rho\in \Phi_+^\re$ and all $n\in\Z_{>0}$.  
\end{enumerate}
\end{Theorem}

The rest of Section~\ref{SRough} is devoted to the proof of Theorem~\ref{THeadIrr}, which goes by induction on $\height(\al)$. To be more precise, we prove the following statements for all $\al\in Q_+$ by induction on $\height(\al)$: 
\begin{enumerate}
\item[{\rm (1)}] For each $\rho\in\Phi_+^\re$ with $\height(\rho)\leq\height(\al)$ there exists a unique up to isomorphism irreducible $R_\rho$-module $L_\rho$ which satisfies the property (Cus1). Moreover, $L_\rho$ then also satisfies the property (vi) of Theorem~\ref{THeadIrr} if $\height(n\rho)\leq\height(\al)$.

\item[{\rm (2)}] For each $n\in\Z_{\geq 0}$ with $\height(n\de)\leq \height(\al)$ there exist  irreducible $R_{n\de}$-modules $\{L(\umu)\mid \umu\in\Par_n\}$ which satisfy the property (Cus2).
\item[{\rm (3)}] The standard modules $\De(M,\umu)$ for all $(M,\umu)\in\Pi(\al)$, defined as in (\ref{EStand}) using the modules from (1) and (2), satisfy the properties (i)--(v) of Theorem~\ref{THeadIrr}. 
\end{enumerate}

The induction starts with $\height(\al)=0$, and for $\height(\al)=1$ the theorem is also clear since $R_{\al_i}$ is a polynomial algebra, which has only the trivial representation $L_{\al_i}$. The inductive assumption will stay valid throughout Section~\ref{SRough}.

\subsection{Irreducible heads}

In the following proposition, we exclude the cases where the standard module is either of the form $L_\rho^{\circ n}$ for a real root $\rho$, or is  imaginary of the form $L(\ula)$. The excluded cases will be dealt with in this Sections~\ref{SSIm}, \ref{SSCusp} and \ref{SSPower}.

\begin{Proposition} \label{PHeadIrr} 
Let $(M,\mu)\in \Pi(\al)$, 
and suppose that 
there are integers $a\neq b$ such that $m_a\neq 0$ and $m_b\neq 0$. 
\begin{enumerate}
\item[{\rm (i)}] 
$
\Stand(M,\umu)
$ has an irreducible head; denote this irreducible module $L(M,\umu)$. 

\item[{\rm (ii)}] If $(M,\umu)\neq (N,\unu)$, then $L(M,\umu)\not\cong L(N,\unu)$. 

\item[{\rm (iii)}] $L(M,\umu)^\circledast\simeq L(M,\umu)$.  

\item[{\rm (iv)}] $[\Stand(M,\umu):L(M,\umu)]_q=1$, and $[\Stand(M,\umu):L(N,\unu)]_q\neq 0$ implies $(N,\unu)\leq (M,\umu)$. 

\item[{\rm (v)}] $\Res_{|M|}L(M,\umu)\simeq L_{M,\umu}$ and $\Res_{|N|}L(M,\umu)\neq 0$ implies $N\leq M$.  
\end{enumerate}
\end{Proposition}

\begin{proof}
(i) and (v) If $L$ is an irreducible quotient of $\De(M,\umu)=\Ind_{|M|}L_{M,\umu}$, then by adjointness of $\Ind_{|M|}$ and $\Res_{|M|}$ and the irreducibility of the $R_{|M|}$-module $L_{M,\umu}$, which holds by the inductive assumption, we conclude  
that  $L_{M,\umu}$ is a submodule of $\Res_{|M|} L$. 
On the other hand, by Proposition~\ref{P1}(i) the multiplicity of $L_{M,\umu}$ in $\Res_{|M|} \De(M,\umu)$ is one, so (i) follows. Note that we have also proved the first statement in (v), while the second statement in (v) follows from Proposition~\ref{P1}(ii) and the exactness of the functor $\Res_{|M|}$. 

(iv) By (v), $\Res_{|N|}L(N,\unu)\cong L_{N,\unu}\neq 0$. Therefore, if $L(N,\unu)$ is a composition factor of $\Stand(M,\umu)$, then $\Res_{|N|}\Stand(M,\umu)\neq 0$ by exactness of $\Res_{|N|}$. By Proposition~\ref{P1}, we then have $N\leq M$ and the first equality in (iv). If $N<M$, then $(N,\unu)<(M,\umu)$. If $N=M$, and $\unu\neq \umu$, then we get a contribution of $L_{N,\unu}$ into $\Res_{|M|}\Stand(M,\umu)$, which contradicts (v). 

(ii) If $L(M,\umu)\cong L(N,\unu)$, then we deduce from (iv) that $(M,\umu)\leq (N,\unu)$ and $(N,\unu)\leq (M,\umu)$, whence $(M,\umu)=(N,\unu)$. 

(iii) follows from (v) and Lemma~\ref{LLMMuSeldDual}. 
\end{proof}

\subsection{Imaginary modules}\label{SSIm}
In this subsection we assume that $\al=n\de$ for some $n\in\Z_{\geq 0}$. Then Proposition~\ref{PHeadIrr}, yields $|\Pi(\al)|-|\Par_n|$ (pairwise non-isomorphic) irreducible modules, namely the modules $L(M,\umu)$ corresponding to the root partitions $(M,\umu)$ such that $m_a\neq 0$ for some $a\neq 0$. Let us label the remaining $|\Par_n|$ irreducible $R_{n\de}$-modules by the elements of $\Par_n$ in {\em some} way, cf. Lemma~\ref{LAmount}. So we get  irreducible  $R_{n\de}$-modules 
$\{L(\umu)\mid\umu\in\Par_n\}$, and then $\{L(M,\umu)\mid (M,\umu)\in\Pi(\al)\}$ is a complete and irredundant system of irreducible $R_\al$-modules up to isomorphism. Our next goal is Lemma~\ref{LMcNamaraImag} which proves that the modules $\{L(\umu)\mid\umu\in\Par_n\}$ are imaginary in the sense of (Cus2). 

We need some terminology. Let $(M,\umu)$ be a root partition.
We say that a real root $\rho_a$ (resp. an imaginary root $m_0\de$) {\em appears in the support of} $M$ if $m_a>0$ (resp. $m_0>0$). Let $\kappa$ be the largest root appearing in the support of $M$, and $\beta\succeq \kappa$. Note that if $\be$ is real then $L_\be\circ \De(M,\umu)$ is, up to a degree shift, a standard module again. If $\be=n\de$ is imaginary, $\unu\in\Par_n$, and $\kappa$ is real, then $L(\unu)\circ \De(M,\umu)$ is again a standard module. 

\begin{Lemma} \label{LMcNamaraImag}
Let  $\ula\in\Par_n$. Suppose that $\be,\ga\in Q_+\setminus\Phi_+^\im$ are non-zero elements such that $n\de=\be+\ga$ and $\Res_{\be,\ga}L(\ula)\neq 0$. Then $\beta$ is a sum of real roots less than $\de$ and $\ga$ is a sum of real roots greater than $\de$. 
\end{Lemma}
\begin{proof}
We prove that $\beta$ is a sum of real roots less than $\de$, the proof that $\ga$ is a sum of real roots greater than $\de$ being similar. 
Let $L(M,\umu)\boxtimes L(N,\unu)$ be an irreducible submodule of $\Res_{\be,\ga} L(\ula)\neq 0$, so that $(M,\umu)\in\Pi(\be)$ and $(N,\unu)\in\Pi(\ga)$. Note that $\height(\be),\height(\ga)<\height(\al)$, so the modules $L(M,\umu), L(N,\unu)$ are defined by induction.

Let $\chi$ be the largest root appearing in the support of $M$. If $\chi\leq \de$, then, since $\be$ is not an imaginary root, we conclude that $\beta$  is a sum of real roots less than $\de$. So we may assume that $\chi\in\Phi^\re_{\succ}$. 
 Moreover, $\Res^\be_{\chi,\be-\chi}L(M,\umu)\neq 0$, and hence $\Res_{\chi,\ga+\be-\chi}L(\ula)\neq0$. So we may assume from the beginning that $\be\in\Phi^\re_{\succ}$ and $L(M,\umu)\simeq L_\be$. Moreover, we may assume that $\be$ is the largest possible real root for which $\Res_{\be,\ga} L(\ula)\neq 0$. 

Now, let $\kappa$ be the largest root appearing in the support of $N$. If $\kappa$ is a real root, we have the cuspidal module $L_\kappa$. If $\kappa$ is imaginary, then let us denote by $L_\kappa$ the module $L(\unu)$. 
Then we have a non-zero map $L_\be\boxtimes L_\kappa\boxtimes V\to \Res_{\be,\kappa,\ga-\kappa}L(\ula)$, for some non-zero $R_{\ga-\kappa}$-module $V$. By adjunction, this yields a non-zero map
$$
f: (\Ind_{\be,\kappa}^{\be+\kappa} L_\be\boxtimes L_\kappa)\boxtimes V\to \Res_{\be+\kappa,\ga-\kappa}L(\ula)
$$

If $\kappa=\ga$ note that $\be\neq \ga$, since it has been assumed that $\be,\ga\not\in\Phi_+^\im$. Now we conclude that $\be\prec\ga$, for otherwise $L(\ula)$ is a quotient of the standard module $L_\be\circ L_\ga$, which contradicts the definition of the imaginary module $L(\ula)$. Now, since $n\de=\be+\kappa$, we have by (Con3) that $\be\prec n\de\prec\ga$, in particular $\be\prec n\de\preceq\de$ as desired. 

Next, let $\kappa\neq\ga$, and pick a composition factor  $L(M',\umu')$ of $\Ind_{\be,\kappa}^{\be+\kappa} L_\be\boxtimes L_\kappa$, which is not in the kernel of $f$. By the assumption on the maximality of $\beta$, every root $\kappa'$ in the support of $M'$ satisfies $\kappa'\preceq \be$. 
Thus $\be+\kappa$ is a sum of roots $\preceq \be$. Now 
(Con2) implies that $\kappa\preceq \be$, and so by adjointness, $L(\ula)$ is a quotient of the standard module $L_\beta\circ\De(N,\unu)$, which is a contradiction. 
\end{proof} 

We now establish a useful property of imaginary modules: 

\begin{Lemma} \label{LTensImagIsImag}
Let $\umu\in \Par_r$ and $\unu\in\Par_s$ with $r+s=n$. Then all composition factors of $L(\umu)\circ L(\unu)$ are of the form $L(\uka)$ for $\uka\in \Par_n$. 
\end{Lemma}
\begin{proof}
Let $L(K,\uka)$ be a composition factor of $L(\umu)\circ L(\unu)$. We need to prove that $k_a=0$ for all $a\neq 0$, i.e. $L(K,\uka)=L(\uka)$. If this is not the case, there is $a>0$ with $k_a\neq 0$. Pick the smallest such $a$, and set $(K',\uka'):=(K,\uka)'_a$, see (\ref{EM'}). By Proposition~\ref{PHeadIrr}(v), we have that $\Res_{|K|}L(K,\uka)\neq 0$, so $\Res_{|K|}(L(\umu)\circ L(\unu))\neq 0$. We apply the Mackey Theorem to conclude that the last module has a filtration with factors of the form
$$
\Ind_{\la_1,\la_2;\underline{\ga}}^{k_a\rho_a;|K'|} V,
$$
where $k_a\rho_a=\la_1+\la_2$, $\underline{\ga}$ is a refinement of $|K'|$, and  
$$\Res_{\la_1,r\de-\la_1}L(\umu)\neq 0\neq \Res_{\la_2,s\de-\la_2}L(\unu).$$ 
By the inductive assumption, we know that $L(\umu)$ and $L(\unu)$ satisfy (Cus2), i.e. $\la_1$ and $\la_2$ are either imaginary roots or a sum of the roots of the form $\rho_b$ with $b<0$. In either case, $\la_1$ and $\la_2$ are sums of the roots less than $\rho_a$, and then so is $k_a\rho_a$. This contradicts (Con1). 
\end{proof}

\subsection{Cuspidal modules}\label{SSCusp}
Throughout this subsection we assume that $\al=\rho_n\in\Phi_+^\re$ for some $n\neq 0$.
Let $(M,\umu)\in \Pi(\al)$ be a root partition of $\al$.
There is a {\em trivial} root partition, denoted $(\al)$, and defined as $(\al)=(M,\emptyset)$, where $m_n=1$, and $m_a=0$ for all $a\neq n$. 
Proposition~\ref{PHeadIrr} yields $|\Pi(\al)|-1$ irreducible $R_\al$-modules, namely the ones which correspond to the {\em non-trivial}\, root partitions $(M,\mu)$. We define the cuspidal module $L_\al$ to be the missing irreducible $R_\al$-module, cf. Lemma~\ref{LAmount}. Then, of course, we have that $\{L(M,\umu)\mid(M,\umu)\in\Pi(\al)\}$ is a complete and irredundant system of irreducible $R_\al$-modules up to isomorphism. We now prove that $L_\al$ satisfies the property (Cus1) and is uniquely determined by it. To be more precise:

\begin{Lemma} \label{LMcNamara}
If  $\be,\ga\in Q_+$ are non-zero elements such that $\al=\be+\ga$ and $\Res_{\be,\ga}L_\al\neq 0$, then $\beta$ is a sum of roots less than $\al$ and $\ga$ is a sum of roots greater than $\al$. Moreover, this property characterizes $L_\al$ among the irreducible $R_\al$-modules uniquely up to isomorphism and degree shift.   
\end{Lemma}
\begin{proof}
We prove that $\beta$ is a sum of roots less than $\al$, the proof that $\ga$ is a sum of roots greater than $\al$ being similar. Let $L(M,\umu)\boxtimes L(N,\unu)$ be an irreducible submodule of $\Res_{\be,\ga} L_\al$, so that $(M,\umu)\in\Pi(\be)$ and $(N,\unu)\in\Pi(\ga)$. Let $\chi$ be the largest root appearing in the support of $M$.  Then $\Res_{\chi,\be-\chi}L(M,\umu)\neq 0$, and hence $\Res_{\chi,\ga+\be-\chi}L_\al\neq0$. If we can prove that $\chi$ is a sum of roots less than $\al$, then by (Con1), (Con3), $\chi$ is a root less than $\al$,  whence, by the maximality of $\chi$, we have that $\beta$ is a sum of roots less than $\al$. 
So we may assume from the beginning that $\be$ is a root and $L(M,\umu)=L_\be$ (if $\be$ is imaginary, $L_\be$ is interpreted as $L(\umu)$). Moreover, we may assume that $\be$ is the largest possible root for which $\Res_{\be,\ga} L_\al\neq 0$. 

Now, let $\kappa$ be the largest root appearing in the support of $N$. If $\kappa$ is a real root, we have the cuspidal module $L_\kappa$. If $\kappa$ is imaginary, then we interpret $L_\kappa$ as $L(\unu)$. Then we have a non-zero map 
$$L_\be\boxtimes L_\kappa\boxtimes V\to \Res_{\be,\kappa,\ga-\kappa}L_\al,$$ 
for some $0\neq V\in\mod{R_{\ga-\kappa}}$. By adjunction, this yields a non-zero map
$$
f: (\Ind_{\be,\kappa} L_\be\boxtimes L_\kappa)\boxtimes V\to \Res_{\be+\kappa,\ga-\kappa}L_\al.
$$

If $\kappa=\ga$, then we must have $\be\prec\ga$, for otherwise $L_\al$ is a quotient of the standard module $L_\be\circ L_\ga$, which contradicts the definition of the cuspidal module $L_\al$. Now, since $\al=\be+\kappa$, we have by (Con1) that $\be\prec\al\prec\ga$, in particular $\be\prec\al$ as desired. 

Next, let $\kappa\neq\ga$, and pick a composition factor  $L(M',\umu')$ of $\Ind_{\be,\kappa}^{\be+\kappa} L_\be\boxtimes L_\kappa$, which is not in the kernel of $f$. By the assumption on the maximality of $\beta$, every root $\kappa'$ in the support of $M'$ satisfies $\kappa'\preceq \be$. 
Thus $\be+\kappa$ is a sum of roots $\preceq \be$. If $\be$ and $\kappa$ are not both imaginary, then (Con2) implies that $\kappa\preceq \be$, and so by adjointness, $L_\al$ is a quotient of the standard module $L_\beta\circ\De(N,\unu)$, which is a contradiction. 

If $\be$ and $\kappa$ are both imaginary, then $\De(N,\unu)=L(\unu)\circ\De(N',\emptyset)$ for $N'$ such that a maximal root appearing in the support of $N'$ is of the form $\rho_a$ with $a<0$. In this case, we have by adjunction that $L_\al$ is a quotient of $L(\umu)\circ L(\unu)\circ L(N',\emptyset)$. It now follows from Lemma~\ref{LTensImagIsImag} that $L_\al$ is a quotient of the standard module of the form $L(\ula)\circ L(N',\emptyset)$ for some composition factor $L(\ula)$ of $L(\umu)\circ L(\unu)$, so we get a contradiction again, since $L_\al$ is cuspidal. 

The second statement of the lemma is clear since, in view of Proposition~\ref{PHeadIrr}(v) and (Con1), the irreducible modules $L(M,\umu)$, corresponding to non-trivial root partitions $(M,\umu)\in \Pi(\al)$, do not satisfy the property (Cus1). 
\end{proof}

\subsection{Powers of cuspidal modules}\label{SSPower}
Assume finally that $\al=n\rho$ for some $\rho\in\Phi_+^\re$ and $n\in\Z_{>1}$.

\begin{Lemma} \label{LCuspPower} 
The induced module $L_\rho^{\circ n}$ is irreducible for all $n\in \Z_{>0}$. 
\end{Lemma}
\begin{proof}
In view of Proposition~\ref{PHeadIrr}, we have the irreducible modules $L(M,\umu)$ for all root partitions $(M,\umu)\in\Pi(\al)$, except for $(N,\unu)=(\rho^n)$ for which $\De(N,\unu)=L_\rho^{\circ n}$. By (Con1), we have that $N\leq M$ for all $(M,\umu)\in\Pi(\al)$, and if $M=N$, then $(M,\umu)=(N,\unu)$. By Proposition~\ref{PHeadIrr}(v), we conclude that 
$L_\rho^{\circ n}$ has only one composition factor $L$ appearing with certain multiplicity $c(q)\in\A$, and such that $L\not\cong L(M,\umu)$ for all $(M,\umu)\in\Pi(\al)\setminus\{(N,\unu)\}$. Finally, by Corollary~\ref{CPowerIrr}, we conclude that $L_\rho^{\circ n}\cong L$. 
\end{proof}

The proof of Theorem~\ref{THeadIrr} is now complete.

\subsection{\boldmath Reduction modulo $p$}
In this section we work with two fields: $F$ of characteristic $p>0$ and $K$ of characteristic $0$. We use the corresponding indices to distinguish between the two situations. Given an irreducible $R_\al(K)$-module $L_K$ for a root partition $\pi\in\Pi(\al)$ we can pick a (graded) $R_\al(\Z)$-invariant lattice $L_\Z$ as follows: 
pick a homogeneous weight vector $v\in L_K$ and set $L_\Z:=R_\al(\Z)v$. 
The lattice $L_\Z$ can be used to {\em reduce modulo $p$}:
$$
\bar L:=L_\Z\otimes_\Z F.
$$
In general, the $R_\al(F)$-module $\bar L$ depends on the choice of the lattice $L_\Z$. However, we have $\CH \bar L=\CH L_K$, so by linear independence of  characters of irreducible $R_\al(F)$-modules, composition multiplicities of irreducible $R_\al(F)$-modules in $\bar L$ are well-defined. In particular, we have  well-defined {\em decomposition numbers}
$$
d_{\pi,\si}:=[\bar L(\pi):L_F(\si)]_q\qquad (\pi,\si\in \Pi(\al)),
$$
which depend only on the characteristic $p$ of $F$, since prime fields are splitting fields for irreducible modules over KLR algebras.

\begin{Lemma} \label{LMultOneRed}
Let $L_K$ be an irreducible $R_\al(K)$-module and let $\bi=i_1^{a_1}\dots i_b^{a_b}$ be an extremal weight for $L_K$. Let $N$ be the irreducible $\circledast$-selfdual $R_\al(F)$-module defined by
$
N:=\tilde f_{i_k}^{a_k}\dots\tilde f_{i_1}^{a_1}1_F.
$ 
Then $[\bar L:N]_q=1$. 
\end{Lemma}
\begin{proof}
Reduction modulo $p$ preserves formal characters, so the result follows from Corollary~\ref{CExtrNew}. 
\end{proof}

\begin{Proposition} \label{PRedModP}
Let $(M,\umu),(N,\unu)\in \Pi(\al)$. Then $d_{(M,\umu),(N,\unu)}\neq 0$ implies $N\leq M$. In particular, reduction modulo $p$ of any cuspidal module is an irreducible cuspidal module again: $\bar L_{\rho}\simeq L_{\rho,F}$. 
\end{Proposition}
\begin{proof}
By Theorem~\ref{THeadIrr}(v), which holds over any field, we conclude that any composition factor of $\bar L_{\rho}$ is isomorphic to $L_{\rho,F}$ up to a degree shift. Now use Lemma~\ref{LMultOneRed}. 
\end{proof}

We complete this section with a version of the James conjecture for any affine type. In the case where the Cartan matrix $\Car={\tt A}_l^{(1)}$ and $\La=\La_0$, this is equivalent to a block version of the classical James Conjecture, cf. \cite[Section 10.4]{Kbull}. The bound on $p$ is inspired by \cite[(3.4)]{BKW} and \cite[(3.11)]{BKllt}. 

\begin{Conjecture}\label{ConjJamesKLR}
{\em
Let $\al\in Q_+$, and $L_K$ be an irreducible $R_\al(K)$-module which factors through $R_\al^{\La}(K)$. Then reduction modulo $p$ of $L_K$ is irreducible provided $p>(\La,\al)-(\al,\al)/2$. 
}
\end{Conjecture}

In view of Lemma~\ref{LPr}, if $\al=\sum_{i\in I}m_i\al_i$, then every $R_\al$-module certainly factors through $R_\al^\La$ for $\La=\sum_{i\in I}m_i\La_i$, although usually a much smaller $\La$ could be used. 

\subsection{Cuspidal modules and dual PBW bases}\label{SSCMDRE}
Recall the $Q_+$-graded $\A$-algebras $\f_\A^*$ and $\f_\A$ and  $\Q(q)$-algebras $\f^*$ and $\f$. 
Suppose that we are given elements
\begin{equation}\label{ERootElements}
\{E_\rho^*\in(\f_\A^*)_\rho\mid\rho\in\Phi_+^\re\}\cup\{E_\ula^*\in(\f_\A)_{|\ula|\de}\mid\ula\in\Par\}.
\end{equation}
Recalling the notation (\ref{EOrderRoots}), for a root partition $(M,\umu)$ we then  define the corresponding {\em dual PBW monomial}
$$
E^*_{M,\umu}:=(E_{\rho_1}^*)^{m_1}(E_{\rho_2}^*)^{m_2}\dots E_\umu^* \dots (E_{\rho_{-2}}^*)^{m_{-2}}(E_{\rho_{-1}}^*)^{m_{-1}}\in\f_\A^*.
$$
We say that (\ref{ERootElements}) is a {\em dual PBW family}  if the following properties are satisfied: 
\begin{enumerate}
\item[{\rm (i)}] (`convexity') if $\be\succ\ga$ are positive roots then 
$E_\ga^* E_\be^*-q^{-(\be,\ga)}E_\be^* E_\ga^*$ is an $\A$-linear combination of elements $E^*_{M,\umu}$ with $(M,\umu)<(\be,\ga)$; here if $\be=n\de$ is imaginary, then $E^*_\be$ is interpreted as $E^*_\umu$ and $(\be,\ga)$ is interpreted as $(\umu,\ga)\in\Pi(\be+\ga)$ for an arbitrary $\umu\in\Par_n$, and similarly for $\ga$ (both $\be$ and $\ga$ cannot be imaginary since then $\be\not\succ\ga$); 

\item[{\rm (ii)}] (`basis') $\{E^*_{M,\umu}\mid(M,\umu)\in\Pi(\al)\}$ is an $\A$-basis of $(\f^*_\A)_\al$ for all $\al\in Q_+$; 

\item[{\rm (iii)}] (`orthogonality') $$(E^*_{M,\umu},E^*_{N,\unu})=\de_{M,N}(E^*_{\umu},E^*_{\unu})\prod_{n\in\Z_{\neq 0}}((E_{\rho_n}^*)^{m_n},(E_{\rho_n}^*)^{m_n});$$

\item[{\rm (iv)}] (`bar-triangularity') $\barinv^*(E^*_{M,\umu}) = E^*_{M,\umu} +$ an $\A$-linear combination of PBW monomials $E^*_{N,\unu}$ for $(N,\unu) < (M,\umu)$.


\end{enumerate}

The following result shows in particular that the elements $E_\rho^*$ of the dual PBW family are determined uniquely up to signs (for a fixed preorder $\preceq$):

\begin{Lemma} \label{LREUnique}
Assume that (\ref{ERootElements}) is a dual PBW family.
\begin{enumerate}
\item[{\rm (i)}] The elements of (\ref{ERootElements}) are $\barinv^*$-invariant.
\item[{\rm (ii)}] Suppose that we are given another family $\{{}'E_\rho^*\in(\f_\A^*)_\rho\mid\rho\in\Phi_+^\re\}\cup\{{}'E_\ula^*\in(\f_\A)_{|\ula|\de}\mid\ula\in\Par\}$ of $\barinv^*$-invariant elements which satisfies the basis and orthogonality properties. 
Then $E_\rho^*=\pm{}'E_\rho^*$ for all $\rho\in\Phi_+^\re$, and for any $\umu\in\Par_n$, we have that $E_\umu^*$ is an $\A$-linear combination of elements ${}'E^*_\unu$ with $\unu\in\Par_n$. 
\end{enumerate}
\end{Lemma}
\begin{proof}
(i) The convexity of $\preceq$ implies that for $\rho\in\Phi_+^\re$ the trivial root partition $(\rho)\in \Pi(\rho)$ is a minimal element of $\Pi(\rho)$ and for $\umu\in\Par_n$ the trivial root partition $(\umu)\in\Pi(n\de)$ is a minimal element of $\Pi(n\de)$. So the bar-triangularity property (iv) implies that  the elements of a dual PBW family are $\barinv^*$-invariant. 

Part (ii) has two statements, one for $E_\rho^*$ with $\rho\in \Phi_+^\re$ and another for $E_\umu^*$ with $\umu\in\Par_n$. Let $\al:=\rho$ in the first statement and $\al:=n\de$ in the second. We prove (ii) by induction on $\height(\al)$, the induction base being clear. For the first statement, by the basis property of dual PBW families, we can write  
\begin{equation}\label{E081212}
{}'E_\rho^*=cE_\rho^*+\sum_{(M,\umu)\in\Pi(\rho)\setminus\{(\rho)\}}c_{M,\umu}E^*_{M,\umu}\qquad(c,c_{M,\umu}\in\A).
\end{equation}

Fix for a moment a non-trivial root partition $(M,\umu)\in \Pi(\rho)$. 
By the orthogonality property of dual PBW families and non-degeneracy of the form $(\cdot,\cdot)$, there is a $\Q(q)$-linear combination $X_{M,\umu}$ of elements $E^*_{M,\unu}$ with $\unu\in \Par_{|\umu|}$ such that $(E^*_{\pi},X_{M,\umu})=\de_{\pi,(M,\umu)}$ for all $\pi\in\Pi(\rho)$. So pairing the right hand side of (\ref{E081212}) with $X_{M,\umu}$ yields $c_{M,\umu}$. On the other hand, by the inductive assumption, each $E^*_{M,\unu}$ is a linear combination of elements of the form ${}'E^*_{M,\ula}$. So using the orthogonality property for the primed family in (ii), we must have $({}'E_\rho^*,X_{M,\umu})=0$ for all non-trivial root partitions $(M,\unu)\in\Pi(\rho)$. So $c_{M,\umu}=0$. 

Thus ${}'E_\rho^*=cE_\rho^*$. Furthermore, the elements ${}'E_\rho^*$ and $E_\rho^*$ belong to the algebra $\f_\A^*$ and are parts of its $\A$-bases, whence  ${}'E_\rho^*=\pm q^n E_\rho^*$. Since both ${}'E_\rho^*$ and $E_\rho^*$ are $\barinv^*$-invariant, we conclude that $n=0$. 

Now, we prove the second statement in (ii). We can write $E_\umu^*$ as 
$$
{}'E_\umu^*=\sum_{\ula\in \Par_{n}} c_\ula\, E^*_\ula +\sum_{(N,\unu)\in\Pi(n\de)\ \text{with}\ |\unu|<n}c_{N,\unu}\,E^*_{N,\unu} \qquad (c_\ula,c_{N,\unu}\in\A). 
$$
Now one shows that all $c_{N,\unu}=0$ by an argument using orthogonality and the inductive assumption as in the previous two paragraphs.
\end{proof}

We now show that under the Khovanov-Lauda-Rouquier categorification (see  Section~\ref{SSKLRCat}), cuspidal systems yield dual PBW families.

\begin{Proposition}\label{PPBWFamily} 
The following set of elements in $\f_\A^*$
\begin{equation}
\{E_\rho^*:=\ga^*([L_\rho])\mid\rho\in\Phi_+^\re\}\cup\{E_\umu^*:=\ga^*([L(\umu)])\mid\ula\in\Par\}
\end{equation}
is a dual PBW family. Moreover, $\{E_\rho^*\mid\rho\in\Phi_+^\re\}$ is a subset of Lusztig's dual canonical basis. 
\end{Proposition}
\begin{proof}
(i) Under the categorification map $\ga^*$, the graded duality $\circledast$ corresponds to $\barinv^*$, so $\ga^*([L])$ is $\barinv^*$-invariant for any $\circledast$-self-dual $R_\al$-module $L$. 
Moreover, under $\ga^*$, the induction product corresponds to the product in $\f_\A^*$, so the convexity condition (i) follows from Theorem~\ref{THeadIrr}(iv) and Lemma~\ref{LDualInd}. Now, note that $E^*_{M,\umu}=\ga^*([\De(M,\umu)])$, so the conditions (ii) and (iv) follow from Theorem~\ref{THeadIrr}(iv) again. It remains to establish the orthogonality property (iii). Under $\ga^*$, the coproduct $r$ corresponds to the map on the Grothendieck group induces by $\Res$. So using (\ref{EDefPropForm}), we get 
$$
(E^*_{M,\umu},E^*_{N,\unu})=\big((E_{\rho_1}^*)^{m_1}
\otimes \dots \otimes E_\umu^* \otimes \dots 
\otimes (E_{\rho_{-1}}^*)^{m_{-1}}, \ga^*([\Res_{|M|}\De(N,\unu)])\big).
$$
By Proposition~\ref{P1}, $\Res_{|M|}\De(N,\unu)=0$ unless $M=N$, and for $M=N$ we have $\Res_{|M|}\De(N,\unu)=L_{\rho_1}^{\circ m_1}\boxtimes\dots \boxtimes L(\unu)\boxtimes\dots \boxtimes L_{\rho_{-1}}^{\circ m_{-1}}$. Since the form $(\cdot,\cdot)$  is symmetric, the orthogonality  follows from the preceding remarks. 

(ii) For symmetric Cartan matrices we can deduce that each $E_\rho^*$ is a dual canonical basis element using Proposition~\ref{PRedModP} and the main result of \cite{VV}. In general, we can argue as follows. It is known that the elements of the dual canonical basis parametrized by the real roots $\rho$ coincide with the corresponding elements of the dual PBW basis, see \cite[Proposition 8.2]{LBraid}. By \cite[Proposition~40.2.4]{Lubook}, the dual PBW basis (with an arbitrary choice of a $\barinv^*$-invariant $\A$-basis of the `imaginary part' $P$, cf. \cite[Section~40.2.3]{Lubook} and \cite{Beck,BCP}) satisfies the properties of Lemma~\ref{LREUnique}(ii). So 
the elements $E_\rho^*$ of our dual PBW family belong to the dual canonical basis up to signs. 
In view of the commutativity of the triangle (\ref{ETriangle}), it now suffices to find for an arbitrary element $v^*$ of the dual canonical basis just one word $\bi\in\words$ such that the coefficient of $\bi$ in $\iota(v^*)$ evaluated at $q=1$ is positive. But this follows from Lemma~\ref{LExtrDCB}. 
\end{proof}

\begin{Remark} 
{\rm 
For certain special convex preorders, which we refer to as {\em Beck preorders}, (dual)  PBW families have been constructed in \cite{Beck,BCP}. Fix a Beck preorder and denote by 
$\{{}'E_\rho^*\in(\f_\A^*)_\rho\mid\rho\in\Phi_+^\re\}\cup\{{}'E_\ula^*\in(\f_\A)_{|\ula|\de}\mid\ula\in\Par\}$ the corresponding dual PBW family.  By Lemma~\ref{LREUnique}(ii),  ${}'E_\rho^*=\pm E^*_\rho$ for all $\rho\in\Phi_+^\re$. In fact, ${}'E_\rho^*= E^*_\rho$ for all $\rho\in\Phi_+^\re$ by Proposition~\ref{PPBWFamily} since the real dual root elements of Beck-Chari-Pressley basis are known to belong to dual canonical basis. 
}
\end{Remark}

\section{Minuscule representations and imaginary tensor spaces}\label{SMinusc}
In this section we study the `smallest' imaginary representations, namely the imaginary representations of $R_\de$. Then we consider induction powers of these minuscule representations, which turn out to play a role of tensor spaces. Denote 
$$
e:=\height(\de).
$$
Throughout the section we assume that our convex preorder $\preceq$ is balanced, see (\ref{EBalanced}). In particular, this implies that $\al_i\succ n\de\succ\al_0$ for all $n\in \Z_{>0}$ and $i\in I'=\{1,\dots,l\}$. So for any imaginary irreducible representation $L$ of $R_{n\de}$,  we conclude using (Cus2) that $\Res_{\al_i,n\de-\al_i}L=0$ for all $i\in I'$, i.e. all weights $\bi=(i_1,\dots,i_d)$ of $L$ have the property that $i_1=0$. 

\subsection{Minuscule representations} \label{SSMinusc}
Note that $|\Par_1|=l$, so there are exactly $l$ imaginary irreducible representations of $R_\de$. We call these representations {\em minuscule}. 
The following lemma shows that a description of minuscule imaginary modules is equivalent to a  description of the irreducible $R_\de^{\La_0}$-modules.

\begin{Lemma} \label{LFactors}
Let $L$ be an irreducible $R_\de$-module. The following are equivalent: 
\begin{enumerate}
\item[{\rm (i)}] $L$ is minuscule imaginary; 
\item[{\rm (ii)}] $L$ factors through to the cyclotomic quotient $R_\de^{\La_0}$;
\item[{\rm (iii)}] we have $i_1=0$ for any weight $\bi=(i_1,\dots,i_e)$ of $L$.
\end{enumerate}
\end{Lemma}
\begin{proof}
By (\ref{Ea_0}), there is exactly one $0$ among the entries $i_1,\dots,i_e$ of an  arbitrary word $\bi\in\words_\de$. Now (ii) and (iii) are equivalent by Lemma~\ref{LPr}.  The implication (i) $\implies$ (iii) follows from the remarks in the beginning of Section~\ref{SMinusc}.
Finally, let $L(M,\umu)$ be an irreducible $R_\de$-module, which is not imaginary, i.e. there is $a\neq 0$ with $m_a\neq 0$. Then, since $\sum_{a\in \Z} M_a=\de$, we conclude that there is $a>0$ 
with $m_a\neq 0$. 
Let $a$ be the smallest positive integer with $m_a\neq 0$. Then $\rho_a\in\Phi'_+$, in particular, $j_1\neq 0$ for all weights $\bj=(j_1,\dots)$ of $L(\rho_a)$. In view of Theorem~\ref{THeadIrr}(v), we have $L_{M,\umu}\subseteq \Res_{|M|}L(M,\umu)$. In particular, there is a weight $\bi=(i_1,\dots)$ of $L(M,\umu)$ with $i_1\neq 0$. 
\end{proof}

We always consider $R_\al^{\La_0}$-modules as $R_\al$-modules via $\infl^{\La_0}$. 

\begin{Lemma} \label{L3912}
Let $\be\in \Phi_+'$. The cuspidal module $L_{\de-\be}$ factors through $R_{\de-\be}^{\La_0}$ and it is the only irreducible $R_{\de-\be}^{\La_0}$-module. 
\end{Lemma}
\begin{proof}
Let $(M,\umu)\in\Pi(\de-\be)$. 
In view of Lemma~\ref{LPr}, it suffices to prove that if $(M,\umu)$ is non-trivial in the sense of Section~\ref{SSCusp} then $i_1\neq 0$ for some weight $\bi=(i_1,\dots)$ of $L(M,\umu)$. But if $(M,\umu)$ is non-trivial, then 
there is $a>0$ with $m_a\neq 0$. Take the smallest such $a$. Then $\rho_a\in\Phi_+'$, so $j_1\neq 0$ for all weights $\bj=(j_1,\dots)$ of $L(\rho_a)$. By  Theorem~\ref{THeadIrr}(v), we have $L_{M,\umu}\subseteq \Res_{|M|}L(M,\umu)$. In particular, there is a weight $\bi=(i_1,\dots)$ of $L(M,\umu)$ with $i_1\neq 0$. 
\end{proof}

\begin{Corollary} \label{CGreat} 
The minuscule imaginary modules are exactly 
$$\{L_{\de,i}:=\tilde f_i L_{\de-\al_i}\mid i\in I'\}.$$ 
Moreover, $e_j L_{\de,i}=0$ for all $j\in I\setminus\{ i\}$. Thus, for each $i\in I'$, the minuscule imaginary module $L_{\de,i}$ can be characterized uniquely up to isomorphism as the irreducible $R_\de^{\La_0}$-module such that $i_e=i$ for all weights $\bi=(i_1,\dots,i_e)$ of $L_{\de,i}$. 
\end{Corollary}
\begin{proof}
If $L$ and $L'$ are two minuscule imaginary modules, with $e_iL\neq 0$ and $e_i L'\neq 0$, then by Lemmas~\ref{LFactors} and \ref{L3912}, we have that $\tilde e_iL\cong \tilde e_i L'$, whence $L\cong L'$ by Proposition~\ref{PCryst1}(i). It follows by a counting argument that for each minuscule imaginary module $L$ there exists exactly one $i$ with $e_i L\neq 0$, and then, by Lemma~\ref{L3912}, we must have $\tilde e_i L\cong L_{\de-\al_i}$ and $L\cong \tilde f_i L_{\de-\al_i}$. 
\end{proof}

For each $i\in I'$, we refer to the minuscule module $L_{\de,i}$ described in Corollary~\ref{CGreat} as the minuscule module of {\em color~$i$}. Let 
\begin{equation}\label{EMuJ}
\umu(i):=(\emptyset,\dots,\emptyset,(1),\emptyset,\dots,\emptyset)\in\Par_1\qquad(i\in I')
\end{equation} 
be the $l$-multipartition of $1$ with the partition $(1)$ in the $i$th component. 
We associate to it the minuscule module $L_{\de,i}$:
\begin{equation}\label{EMinLabel}
L(\umu(i)):=L_{\de,i}\qquad(i\in I'). 
\end{equation}

\begin{Lemma} \label{LEpsLDe}
Let $i\in I'$. Then $\eps_i(L_{\de,i})=1$.
\end{Lemma}
\begin{proof}
Otherwise $e_i^2(L_{\de,i})\neq 0$, whence $\La_0-\de+2\al_i$ is a weight of $V(\La_0)$, which is a contradiction. 
\end{proof}

\begin{Remark} \label{ROGen}
{\rm 
The minuscule modules are defined over $\Z$. To be more precise, for each $i\in I'$, there exists an $R_\de(\Z)$-module $L_{\de,i,\Z}$ which is free finite rank over $\Z$ and such that $L_{\de,i,\Z}\otimes F$ is the minuscule imaginary module $L_{\de,i,F}$ over $R_\de(F)$ for any ground field $F$. To construct $L_{\de,i,\Z}$, recall that a prime field is a splitting field for $R_{\al}$. 
Now,  start with the minuscule module $L_{\de,i,\Q}$ over $\Q$, pick any weight vector $v$ and consider the lattice  $L_{\de,i,\Q}:=R_\de(\Z)v$. Then $L_{\de,i,\Z}\otimes_\Z \Q\cong L_{\de,i,\Q}$. To see that  $L_{\de,i,\Z}\otimes_\Z F$ is the minuscule module $L_{\de,i,F}$ over any filed $F$, it suffices to prove that $L_{\de,i,\Z}\otimes_\Z F$ is irreducible. If $L(M,\umu)$ is a composition factor of  $L_{\de,i,\Z}\otimes_\Z F$ with $m_a\neq 0$ for some $a\neq 0$, then we get a contradiction with the definition of an imaginary module. So, taking into account the character information, all composition factors of $L_{\de,i,\Z}\otimes_\Z F$ are of the form $L_{\de,i,F}$. Now, in fact we must have $L_{\de,i,\Z}\otimes_\Z F\simeq L_{\de,i,F}$ using the multiplicity one result from Lemma~\ref{LMultOneRed}.
}
\end{Remark}



\subsection{Imaginary tensor spaces}\label{SSCTS}

The {\em imaginary tensor space of color $i$} is the $R_{n\de}$-module $$M_{n,i}:=L_{\de,i}^{\circ n}.$$ In this definition we allow $n$ to be zero, in which case $M_{0,i}$ is interpreted as the trivial module over the trivial algebra $R_0$.

\begin{Lemma} \label{LMSelfDual}
$
\Mde_n^\circledast\simeq \Mde_n. 
$
\end{Lemma}
\begin{proof}
This comes from  Lemma~\ref{LDualInd} using $(\de,\de)=0$. 
\end{proof}

A composition factor of $M_{n,i}$ is called an imaginary module of {\em color $i$}. 
We remark that by Lemma~\ref{LTensImagIsImag} such composition factor is indeed an imaginary module in the sense of (Cus2). Another application of Lemma~\ref{LTensImagIsImag} now gives:

\begin{Lemma} \label{LD0}
All composition factors of $M_{n_1,1}\circ\dots\circ M_{n_l,l}$ are imaginary.
\end{Lemma}

We next observe that if an irreducible $R_{n\de}$-module $L$ (with $n>0$) is imaginary of color $i\in I'$, then $L$ cannot be imaginary of color $j\in I'$, i.e. the color is well defined.  Indeed, if $L$ is imaginary of color $i$, then  
by (\ref{ECharShuffle}) we have that $\eps_i(L)>0$ while $\eps_j(L)=0$ for any $j\neq i$.

\begin{Lemma} \label{LD1}
Let $i\in I'$ and $n_1,\dots,n_a\in\Z_{>0}$. Set $n:=n_1+\dots+n_a$. Then all composition factors of $\Res_{n_1\de,\dots,n_a\de}M_{n,i}$ are of the form $L^1\boxtimes\dots\boxtimes L^a$ where $L^1,\dots,L^a$ are imaginary of color $i$. 
\end{Lemma}
\begin{proof}
By the Mackey Theorem, $\Res_{n_1\de,\dots,n_a\de}M_{n,i}$ has filtration with factors of the form
$$
\Ind^{\,n_1\de\,;\,\dots\,;\,n_a\de}_{\nu_{11},\dots,\nu_{n1}\,;\,\dots\,;\,\nu_{1a},\dots,\nu_{na}}V,
$$
where $\sum_{m=1}^n{\nu_{mb}}=n_b\de$ for all $b=1,\dots,a$, $\sum_{b=1}^a{\nu_{mb}}=\de$ for all $m=1,\dots,n$, and $V$ is obtained by an appropriate twisting of the module  
$$
(\Res_{\nu_{11},\dots,\nu_{1a}}L_{\de,i})\boxtimes\dots\boxtimes (\Res_{\nu_{n1},\dots,\nu_{na}}L_{\de,i}).
$$
If $\nu_{m1}\neq 0$ and $\nu_{m1}\neq \de$ for some $m$, then by Lemma~\ref{LMcNamaraImag}, we have that $\nu_{m1}$ is a sum of real roots less than $\de$, which leads to a contradiction with $\sum_{m=1}^n{\nu_{m1}}=n_1\de$. So we deduce that $\nu_{m1}= \de$ for $n_1$ different values of $m$, and $\nu_{m1}= 0$ for all other values of $m$. Then $L^1\boxtimes L^2\boxtimes\dots\boxtimes L^a$ is a composition factor of 
$$M_{n_1,i}\boxtimes \Res_{n_2\de,\dots,n_a\de}M_{n-n_1,i},$$ 
and the lemma follows by induction.
\end{proof}

\begin{Corollary} \label{CD1}
Let $i\in I'$ and $n_1,\dots,n_a\in\Z_{\geq 0}$. Set $n:=n_1+\dots+n_a$. If $L$ is an imaginary irreducible $R_{n\de}$-module of color $i$, then all composition factors of $\Res_{n_1\de,\dots,n_a\de}L$ are of the form $L^1\boxtimes\dots\boxtimes L^a$ where $L^1,\dots,L^a$ are imaginary of color $i$. 
\end{Corollary}
\begin{proof}
Follows from Lemma~\ref{LD1}, since by definition $L$ is a composition factor of $M_{n,i}$. 
\end{proof}

\subsection{Reduction to one color}

The goal of this section is to prove:

\begin{Theorem} \label{TD1}
Let $n\in\Z_{\geq 0}$, and suppose that for each $i\in I'$, we have an irredundant family $\{L_i(\la)\mid\la\vdash n\}$ of irreducible imaginary $R_{n\de}$-modules of color $i$. For a multipartition $\ula=(\la^{(1)},\dots,\la^{(l)})\in\Par_n$, define $$L(\ula):=L_1(\la^{(1)})\circ\dots\circ L_l(\la^{(l)}).$$ 
Then $\{L(\ula)\mid\ula\in\Par_n\}$ is a complete and irredundant system of imaginary irreducible $R_{n\de}$-modules. 
\end{Theorem}

We prove the theorem by induction on $n$. 
The induction base is clear. Throughout this section we work under the induction hypothesis.

\begin{Lemma} \label{LD3}
Let $\ula,\umu\in\Par_n$ with $\la^{(i)}\vdash n_i$ for $i=1,\dots,l$. If 
the irreducible $R_{n_1\de,\dots,n_l\de}$-module $L_1(\la^{(1)})\boxtimes\dots\boxtimes L_l(\la^{(l)})$ appears as a composition factor~in 
\begin{equation}\label{ED2}
\Res_{n_1\de,\dots,n_l\de}\,L(\umu),
\end{equation} 
then $\ula=\umu$, and the multiplicity of this composition factor is one. 
\end{Lemma}
\begin{proof}
Let $\mu^{(i)}\vdash m_i$ for $i=1,\dots,l$. By the Mackey Theorem, the module in (\ref{ED2}) has filtration with factors of the form 
\begin{equation}\label{ED3}
\Ind^{\,n_1\de\,;\,\dots\,;\,n_l\de}_{\nu_{11},\dots,\nu_{l1}\,;\,\dots\,;\,\nu_{1l},\dots,\nu_{ll}}V,
\end{equation}
where $\sum_{i=1}^l{\nu_{ij}}=n_j\de$ for all $j\in I'$, $\sum_{j=1}^l{\nu_{ij}}=m_{i}\de$ for all $i\in I'$, and $V$ is obtained by an appropriate twisting of the module  
\begin{equation}\label{ED1}
(\Res_{\nu_{11},\dots,\nu_{1l}}L_{1}(\mu^{(1)}))\boxtimes\dots\boxtimes (\Res_{\nu_{l1},\dots,\nu_{ll}}L_{l}(\mu^{(l)})).
\end{equation}
Assume that the module in (\ref{ED3}) is non-zero. 

Since each $L_{i}(\mu^{(i)})$ is imaginary and $\Res_{\nu_{i1},\dots,\nu_{il}}L_{i}(\mu^{(i)})\neq 0$, it follows by Lemma~\ref{LMcNamaraImag} 
that either $\nu_{i1}=n_{i1}\de$ for some $n_{i,1}\in\Z_{\geq 0}$, or $\nu_{i1}$ a sum of real roots less than $m_i\de$. Since $\sum_{i=1}^l{\nu_{i1}}=n_1\de$, we conclude that the second option is impossible. 
Next, we claim that also each $\nu_{i2}=n_{i2}\de$ for some $n_{i2}\in\Z_{\geq 0}$. Indeed, since $\Res_{\nu_{i1},\dots,\nu_{il}}L_{i}(\mu^{(i)})\neq 0$, we have that  $\Res_{\nu_{i1}+\nu_{i2},m_{i}\de-\nu_{i1}-\nu_{i2}}L_{i}(\mu^{(i)})\neq 0$. By Lemma~\ref{LMcNamaraImag}, either $\nu_{i1}+\nu_{i2}$ is an imaginary root, or it is a sum of real roots less than $m_{i}\de$. Since we already know that the $\nu_{i,1}$ are imaginary roots (or zero), the equality $\sum_{i=1}^l{\nu_{i2}}=n_2\de$ implies that $\nu_{i2}=n_{i2}\de$ for some $n_{i2}\in\Z_{\geq 0}$. Continuing this way, we establish that all $\nu_{ij}$ are of the form $n_{ij}\de$. 

Now, by Corollary~\ref{CD1}, all composition factors of   $\Res_{\nu_{i1},\dots,\nu_{il}}L_{i}(\mu^{(i)})$ are of the form $L_{i}(\mu^{(i1)})\boxtimes\dots\boxtimes L_{i}(\mu^{(il)})$. Then the module in (\ref{ED2})  has filtration with factors of the form 
$$
\big(L_{1}(\mu^{(11)})\circ\dots\circ L_{l}(\mu^{(l1)})\big)\boxtimes\dots\boxtimes \big(L_{1}(\mu^{(1l)})\circ\dots\circ L_{l}(\mu^{(ll)})\big).
$$
By the inductive hypothesis, each $L_{1}(\mu^{(1j)})\circ\dots\circ L_{l}(\mu^{(lj)})$ is irreducible, and  
$$L_{1}(\mu^{(1j)})\circ\dots\circ L_{l}(\mu^{(lj)})\cong L_j(\la^{(j)})$$ 
if and only if $\mu^{(jj)}=\la^{(j)}$ and $\mu^{(ij)}=\emptyset$ for all $i\neq j$. Thus $\nu_{jj}=n_{j}\de$, $\nu_{ij}=0$ for all
$i\neq j$. We conclude that $m_j=n_j$ and  $\mu^{(j)}= \la^{(j)}$ for all $j$. 
\end{proof}

\begin{Corollary} \label{CD5}
The module $L(\ula)$ has simple head; denote it by $L^\ula$. The multiplicity of $L^\ula$ in $L(\ula)$ is one.  
\end{Corollary}
\begin{proof}
If an irreducible module $L$ is in the head of  $L(\ula)$, then by the adjunction of $\Ind$ and $\Res$, we have that $L_1(\la^{(1)})\boxtimes\dots\boxtimes L_l(\la^{(l)})\subseteq \Res_{n_1\de,\dots,n_l\de}L$. Now the result follows from Lemma~\ref{LD3} with $\ula=\umu$. 
\end{proof}

\begin{Corollary} \label{CD4}
If $\ula\neq \umu$, then $L^\ula\not\cong L^\umu$. 
\end{Corollary}
\begin{proof}
Assume that $L^\ula\cong L^\umu$. Then $L^\umu$ is a quotient of $L(\ula)$. By 
the adjunction of $\Ind$ and $\Res$, we have that $L_1(\la^{(1)})\boxtimes\dots\boxtimes L_l(\la^{(l)})\subseteq \Res_{n_1\de,\dots,n_l\de}L^\umu$. In particular, $L_1(\la^{(1)})\boxtimes\dots\boxtimes L_l(\la^{(l)})$ is a composition factor of $\Res_{n_1\de,\dots,n_l\de}L(\umu)$. Now, by Lemma~\ref{LD3}, we have $\ula=\umu$. 
\end{proof}

Now we can finish the proof of Theorem~\ref{TD1}. By counting using Theorem~\ref{THeadIrr}, Lemma~\ref{LD0}, and Corollary~\ref{CD4}, we see that $\{L^\ula\mid\ula\in\Par_n\}$ is a complete and irredundant set of irreducible imaginary $R_{n\de}$-modules. It remains to prove that $L(\umu)$ is irreducible, i.e. $L(\umu)=L^\umu$, for each $\umu$. If $L(\umu)$ is not irreducible, let $L^\ula\not\cong L^\umu$ be an irreducible submodule in the socle of $L(\umu)$, see Corollary~\ref{CD5}. Then there is a nonzero homomorphism $L(\ula)\to L(\umu)$, whence by 
the adjunction of $\Ind$ and $\Res$, we have that $L_1(\la^{(1)})\boxtimes\dots\boxtimes L_l(\la^{(l)})\subseteq \Res_{n_1\de,\dots,n_l\de}L(\umu)$. Now, by Lemma~\ref{LD3}, we have $\ula=\umu$. Theorem~\ref{TD1} is proved.

\subsection{Homogeneous modules}\label{SShomog}  
In the remainder of Section~\ref{SMinusc} we describe the minuscule imaginary modules more explicitly for symmetric (affine) Cartan matrices. 
This is done using the theory of homogeneous representations developed in \cite{KRhomog}, which we review next.
Throughout this subsection we assume that the Cartan matrix $\Car$ is symmetric. As usual, we work with an arbitrary fixed $\al\in Q_+$ of height $d$. 
A graded $R_\al$-module is called {\em homogeneous} if it is concentrated in one degree. 

Let  $\bi\in \words_\al$. We call $s_r\in S_d$ an {\em admissible transposition} for $\bi$ if $\cc_{i_r, i_{r+1}}=0$. 
The {\em weight graph} $G_\al$ is the graph with the set of vertices $\words_\al$, and with $\bi,\bj\in \words_\al$ connected by an edge if and only if $\bj=s_r \bi$ for some admissible transposition $s_r$ for $\bi$. 

Recall from Section~\ref{SSLTN} the Weyl group $W=\lan r_i\mid i\in I\ran$. 
Let $C$ be a connected component of $G_\al$, and $\bi=(i_1,\dots,i_d)\in C$. We set  
$$
w_C:=r_{i_d}\dots r_{i_1}\in W.
$$
Clearly the element $w_C$ depends only on $C$ and not on $\bi\in C$. 
An element $w\in W$ is called {\em fully commutative} if any reduced expression for $w$  can be obtained from any other by using only the Coxeter relations that involve commuting generators, see e.g. \cite{S1}. For an integral weight $\La\in P$, an element $w\in W$ is called {\em $\La$-minuscule} if there is a reduced expression $w=r_{i_l}\dots r_{i_1}$ such that 
$$
\lan r_{i_{k-1}}\dots r_{i_1}\La,\al_{i_k}^\vee\ran=1
\qquad (1\leq k\leq l),
$$
cf. \cite[Section 2]{S2}. By \cite[Proposition 2.1]{S2}, if $w$ is $\La$-minuscule for some $\La\in P$, then $w$ is fully commutative.

A connected component $C$ of $G_\al$ is called {\em homogeneous} (resp. {\em strongly homogeneous}) if for some (equivalently every) $\bi=(i_1,\dots,i_d)\in C$, we have that $r_{i_d}\dots r_{i_1}$ is a reduced expression for a fully commutative (resp. {\em minuscule}) element $w_C\in W$, cf. \cite[Sections 3.2, Definition~3.5, Proposition~3.7]{KRhomog}. In that case, there is an obvious one-to-one correspondence between the elements $\bi\in C$ and the reduced expressions of $w_C$.


\begin{Lemma}\label{LHomComp}   {\rm \cite[Lemma 3.3]{KRhomog}} 
A connected component $C$ of $G_\al$ is homogeneous if and only if for some $\bi=(i_1,\dots,i_d)\in C$ the following condition holds:
\begin{equation}\label{ENC}
\begin{split}
\text{if $i_r=i_s$ for some $r<s$ then there exist $t,u$}
\\  
\text{such that $r<t<u<s$ and $\cc_{i_r,i_t}=\cc_{i_r,i_u}=-1$.}
\end{split}
\end{equation}
\end{Lemma}

The main theorem on homogeneous representations is:

\begin{Theorem}\label{Thomog} {\rm \cite[Theorems 3.6, 3.10, (3.3)]{KRhomog}} 

\begin{enumerate}
\item[{\rm (i)}] Let $C$ be a homogeneous connected component of $G_\al$. Let $L(C)$ be the vector space  concentrated in degree $0$ with basis $\{v_\bi\mid \bi\in C\}$ labeled by the elements of $C$. 
The formulas
\begin{align*}
1_\bj v_\bi&=\de_{\bi,\bj}v_\bi \qquad (\bj\in \words_\al,\ \bi\in C),\\
y_r v_\bi&=0\qquad (1\leq r\leq d,\ \bi\in C),\\
\psi_rv_{\bi}&=
\left\{
\begin{array}{ll}
v_{s_r\bi} &\hbox{if $s_r\bi\in C$,}\\
0 &\hbox{otherwise;}
\end{array}
\right.
\quad(1\leq r<d,\ \bi\in C)
\end{align*}
define an action of $R_\al$ on $L(C)$, under which $L(C)$ is a homogeneous irreducible $R_\al$-module. 
\item[{\rm (ii)}] $L(C)\not\cong L(C')$ if $C\neq C'$, and every homogeneous irreducible $R_\al$-module, up to a degree shift, is isomorphic to one of the modules $L(C)$. 
\item[{\rm (iii)}] If $\be,\ga\in Q_+$ with $\al=\be+\ga$, then $\Res_{\be,\ga}L(C)$ is either zero or irreducible. 
\end{enumerate}
\end{Theorem}

\subsection{Minuscule representations in simply laced types}\label{SSMinSLaced}
Throughout this subsection we assume that the Cartan matrix $\Car$ is symmetric.

\begin{Lemma} \label{Lw(i)}
Let $i\in I'$. Then we can write $\La_0-\de+\al_i=w(i)\La_0$ for a unique $\La_0$-minuscule element $w(i)\in W$.
\end{Lemma}
\begin{proof}
Let $\theta$ be the highest root in the finite root system $\Phi'$. Pick a (unique) minimal length element $u$ of the finite Weyl group $W'$ with $u\theta=\al_i$. Now, take $w(i)=ur_0$. Note that 
\begin{align*}
w(i)(\La_0)&=ur_0(\La_0)=u(\La_0-\al_0)=u(\La_0-\al_0-\theta+\theta)=u(\La_0-\de+\theta)
\\
&=\La_0-\de+u(\theta)=\La_0-\de+\al_i.
\end{align*}
Since the $\al$-string through $\beta$ has length $0$ or $1$ for any distinct roots $\al,\be\in\Phi'$, we deduce that $u$ is $\theta$-minuscule, and the lemma follows.
\end{proof}

By the theory described in Section~\ref{SShomog}, the minuscule element $w(i)$ constructed in Lemma~\ref{Lw(i)} is of the form $w_{C(i)}$ for some strongly homogeneous component $C(i)$ of $G_{\de-\al_i}$. 

\begin{Lemma} \label{L3912_2}
Let $i\in I'$, $d:=e-1=\height(\de-\al_i)$ and $\bj=(j_1,\dots,j_{d})\in C(i)$. Then:
\begin{enumerate}
\item[{\rm (i)}] $j_1=0$;
\item[{\rm (ii)}] $j_d$ is connected to $i$ in the Dynkin diagram, i.e. $\cc_{j_d,i}<0$;
\item[{\rm (iii)}] if $j_b=i$ for some $b$, then there are at least three indices $b_1,b_2,b_3$ such that $b<b_1<b_2<b_3\leq d$ such that $\cc_{i,b_1}=\cc_{i,b_2}=\cc_{i,b_3}=-1$. 
\end{enumerate}  
\end{Lemma}
\begin{proof}
(i)  is clear from the construction of $w(i)$ which always has $r_0$ as the last simple reflection in its reduced decomposition. 

(ii) Let $w(i)=r_{j_d}\dots r_{j_1}$ be a reduced decomposition. By definition of a minuscule element, we conclude that $\lan\La_0-\de+\al_i,\al_{j_d}^\vee\ran<0$, so $\lan\al_i,\al_{j_d}^\vee\ran<0$.


(iii) If $j_b=i$, then, using the definition of a minuscule element and the equality 
$w(i)\La_0=r_{j_d}\dots r_{j_1}\La_0=\La_0-\de+\al_i,$ 
we see that 
$$\lan r_{j_{b+1}}\dots r_{j_d}(\La_0-\de+\al_i),\al_i^\vee\ran=\lan r_{j_b}r_{j_{b-1}}\dots r_{j_1}\La_0,\al_{j_b}^\vee\ran=-1.$$ This implies (iii), since $\lan\La_0-\de+\al_i,\al_i^\vee\ran=2$.  
\end{proof}

\begin{Corollary} 
Let $i\in I'$. Then the cuspidal module $L_{\de-\al_i}$ is the homogeneous module $L(C(i))$. 
\end{Corollary}
\begin{proof}
By Lemmas~\ref{L3912_2}(i) and~\ref{LPr}, the module $L(C(i))$ factors through $H_{\de-\al_i}^{\La_0}$. So $L(C(i))\cong L_{\de-\al_i}$ by Lemma~\ref{L3912}. 
\end{proof}

\begin{Proposition} \label{PMinSL}
Let $i\in I'$. The set of concatenations
$$
C_i:=\{\bj i\mid \bj\in C(i)\}
$$ 
is a homogeneous component of $G_\de$, and the corresponding homogeneous $R_{\de}$-module $L(C_i)$ is isomorphic to the minuscule imaginary module $L_{\de,i}$. 
\end{Proposition}
\begin{proof}
By Lemmas~\ref{LHomComp} and \ref{L3912_2}(ii),(iii), we have that $C_i$ is a homogeneous connected component of $G_\de$. By Lemmas~\ref{L3912_2}(i) and \ref{LPr}, the corresponding homogeneous representation $L(C_i)$ factors through to $R_\de^{\La_0}$, and so it must be one of the minuscule representations $L_{\de,1},\dots,L_{\de,l}$, see Corollary~\ref{CGreat}. Finally, by the second statement in Corollary~\ref{CGreat}, we must have $L(C_i)\cong L_{\de,i}$. 
\end{proof}

\begin{Example} \label{ETypeA}
{\rm 
Let  $\Car={\tt A}_l^{(1)}$ and $i\in I'$. Then $L_{\de,i}$ is the homogeneous irreducible $R_{\de}$-module with character 
$$
\CH L_{\de,i}=0\big((12\dots i-1)\circ (l,l-1,\dots,i+1)\big)i.
$$
For example, $L_{\de,1}$ and $L_{\de,l}$ are $1$-dimensional with characters
$$
\CH L_{\de,1}=(0,l,l-1,\dots,1),\quad
\CH L_{\de,l}=(01\dots l),
$$
while for $l\geq 3$, the module $L_{\de,l-1}$ is $(l-2)$-dimensional with character
$$
\CH L_{\de,l-1}=\sum_{r=0}^{l-3}(0,1,\dots,r, l,r+1,\dots,l-1).
$$
}
\end{Example}

\section{More on cuspidal modules}\label{SCusp}
In this section we first work again with an arbitrary convex preorder $\preceq$, and then in subsections~\ref{SS+} and \ref{SS-} we assume that the preorder is balanced.

\subsection{Minimal pairs} Let $\rho\in\Phi_+^\re$. A pair of positive roots $(\be,\ga)$ is called a {\em minimal pair} for $\rho$ if 
\begin{enumerate}
\item[{\rm (i)}] $\be+\ga=\rho$ and $\be\succ\ga$;
\item[{\rm (ii)}] for any other pair $(\be',\ga')$ satisfying (i) we have $\be'\succ\be$ or $\ga'\prec\ga$. 
\end{enumerate}
In view of convexity, $(\be,\ga)$ is a minimal pair for $\rho$ if and only if $(\be,\ga)$ is a minimal element of $\Pi(\rho)\setminus\{(\rho)\}$. A minimal pair $(\be,\ga)$ is  called {\em real} if both $\be$ and $\ga$ are real roots.

\begin{Lemma} \label{LPBW2}
Let $\rho\in\Phi_+^\re$ and $(\be,\ga)$ be a minimal pair for $\rho$. If $L$ is a composition factor of the standard module $\De(\be,\ga)=L(\beta)\circ L(\gamma)$, then $L\cong L(\be,\ga)$ or $L\cong L_\rho$. 
\end{Lemma}
\begin{proof}
Use the minimality of $(\be,\ga)$ in $\Pi(\rho)\setminus\{(\rho)\}$  and Theorem~\ref{THeadIrr}(iv). 
\end{proof}

\begin{Remark} 
{\rm 
Let $(\be,\ga)$ be a {\em real}\, minimal pair for $\rho\in\Phi_+^\re$. Denote 
$$
p_{\beta,\gamma} := \max\,\{n \in \Z_{\geq 0}\mid \beta - n \gamma \in
  \Phi_+\}.
$$
The argument as in the proof of \cite[Theorem 4.2]{BKM} shows that in the Grothendieck group we have
\begin{equation}\label{ET4.2}
[L_\ga\circ L_\de]-q^{-(\be,\ga)}[L_\be\circ L_\ga]=q^{-p_{\be,\ga}}(1-q^{2(p_{\be,\ga}-(\be,\ga))})[L_\rho].
\end{equation}
So one can compute the character of the cuspidal module $L_\rho$ by induction on $\height(\rho)$, provided $\rho$ possesses a real minimal pair, cf. Lemma~\ref{LRealMP} below. 
}
\end{Remark}

\begin{Remark} 
{\rm 
By Lemma~\ref{LPBW2}, we can write in the Grothendieck group
$$
[L_\be\circ L_\ga]=[L(\be,\ga)]+m(q)[L_\rho].
$$
Now, by Lemma~\ref{LDualInd}, we also have
$$
[L_\ga\circ L_\be]=q^{-(\be,\ga)}[L(\be,\ga)]+q^{-(\be,\ga)}m(q^{-1})[L_\rho].
$$
So (\ref{ET4.2}) implies 
$$q^{-(\be,\ga)}(m(q^{-1})-m(q))=q^{-p_{\be,\ga}}(1-q^{2(p_{\be,\ga}-(\be,\ga))}),$$  whence
$$
m(q)-m(q^{-1})=q^{p_{\be,\ga}-(\be,\ga)}-q^{(\be,\ga)-p_{\be,\ga}}.
$$
Now, assume that the Cartan matric $\Car$ is symmetric. Then by the main result of \cite{VV}, we have that $m(q)\in q\Z[q]$, and so the last equality implies
\begin{equation}\label{EMultiplicity}
m(q)=q^{p_{\be,\ga}-(\be,\ga)},
\end{equation}
i.e. there is a short exact sequence 
\begin{align}\label{sesoneas}
0 \longrightarrow L_\rho\langle p_{\beta,\gamma}-(\beta,\gamma)\rangle \longrightarrow L_\beta\circ L_\gamma
\longrightarrow L(\be,\ga) \longrightarrow 0.
\end{align}
Note that for symmetric $\Car$ we always have $p_{\be,\ga}=0$ and $p_{\beta,\gamma}-(\beta,\gamma)=1$. 
}
\end{Remark}

We conjecture that this also holds in non-simply laced affine types (a similar result for all finite types is established in \cite[Theorem 4.7]{BKM}): 

\begin{Conjecture} \label{ConjLT}
{\em For non-symmetric $\Car$, let $\rho\in\Phi_+^\re$, and $(\be,\ga)$ be a real minimal pair for $\rho$. Then there still is a short exact sequence of the form (\ref{sesoneas}). 
}
\end{Conjecture}


\begin{Example} \label{ExMinPair}
{\rm 
Let $n\in\Z_{>0}$ and $i\in I'$. Assume that the preorder is balanced. 

(i) If $\rho=n\de+\al_i$, then $(\al_i+(n-1)\de,\de)$ is a minimal pair for $\rho$. 

(ii) If $n>1$ and $\rho=n\de-\al_i$, then $(\de,(n-1)\de-\al_i)$ is a minimal pair for~$\rho$. 
}
\end{Example}


\begin{Lemma} \label{LRealMP}
Assume that the preorder is balanced. Let $\rho$ be a non-simple positive root. 
Then there exists a real minimal pair for $\rho$, unless $\rho$ is of the form $n\de\pm\al_i$. 
\end{Lemma}
\begin{proof}
If $\rho\in\Phi^\re_{\succ}$ is not of the form $n\de+\al_i$, then we can always write $\rho$ as a sum of two roots in $\Phi^\re_{\succ}$, and so there exists a real minimal pair for $\rho$. 

If $\rho\in\Phi^\re_{\prec}$ is not of the form $n\de-\al_i$ and $n\geq 2$, then we can write $\rho$ as a sum of  two roots in $\Phi^\re_{\prec}$, and so again there exists a real minimal pair for~$\rho$. Finally, in the special case where $\rho$ is a non-simple root of the form $\de-\al$ for $\al\in\Phi_+'$, by an argument of  \cite[Lemma 2.1]{McN} we can write $\rho$ as a sum of two real roots, which implies the result. 
\end{proof}

In view of the lemma, the cuspidal modules corresponding to the roots of the form $n\de\pm\al_i$ play a special role. In Sections~\ref{SS+} and \ref{SS-} we will investigate them in detail. 

\subsection{\boldmath Cuspidal modules $L_{n\de+\al_i}$} \label{SS+}
We continue to assume (until the end of the paper) that  the convex preorder $\preceq$ is balanced. 
We will now use a slightly different notation for the root partitions. For example, if $(M,\umu)$ is such that $m_1=2,m_2=1, m_0=1, m_{-3}=1,$ all other $m_a=0$, and $\umu=\umu(i)$ as in (\ref{EMuJ}), then we write $(M,\umu)=(\rho_1,\rho_1,\rho_2,\de^{(i)},\rho_{-3})$. 

Fix $i\in I'$. In this section we consider the cuspidal modules corresponding to the real roots of the form $n\de+\al_i$ for $i\in I'$. Fix also an extremal weight
\begin{equation}\label{EExtrIm}
\bi=i_1^{a_1}\dots i_k^{a_k}
\end{equation}
of the minuscule imaginary module $L_{\de,i}$, see Section~\ref{SSCOES}. 
Recall from Corollary~\ref{CGreat} and Lemma~\ref{LEpsLDe} that $i_k=i$ and $a_k=1$. We will use 
the concatenations $\bi^n\in\words_{n\de}$, $\bi^ni\in\words_{n\de+\al_i}$ and also the special weight 
$$
\bi^{\{n\}}:=
i_1^{na_1}\dots i_{k-1}^{na_{k-1}} i^{n+1}\in\words_{n\de+\al_i}.
$$

\begin{Proposition}\label{PCuspSpecial+}
Let $i\in I'$, $n\in\Z_{>0}$, $\al=n\de+\al_i$, and $\be=(n-1)\de+\al_i$. Then:
\begin{enumerate}
\item[{\rm (i)}] The standard module $\De(\be,\de^{(i)})=L_{\be}\circ L_{\de,i}$ has composition series of length two with head $L(\be,\de^{(i)})$ and socle $L_{\al}\lan (\al_i,\al_i)/2\ran$. 
\item[{\rm (ii)}] We have 
$$\CH L_\al=\frac{1}{q_i-q_i^{-1}}\big((\CH L_\be)\circ (\CH L_{\de,i})-(\CH L_{\de,i})\circ (\CH L_\be)\big).$$ 
\item[{\rm (iii)}] We have 
$$\CH L_\al=\frac{1}{q_i-q_i^{-1}}\sum_{m=0}^n (-1)^m(\CH L_{\de,i})^{\circ m}\circ i \circ (\CH L_{\de,i})^{\circ (n-m)}.$$
\item[{\rm (iv)}] The weight $\bi^{\{n\}}$ is an extremal weight of $L_\al$. 
\end{enumerate}
\end{Proposition}
\begin{proof}
We apply induction on $n$. Consider the induced modules 
$W_1:=L_{\be}\circ L_{\de,i}$ and $W_2:=L_{\de,i}\circ L_{\be}.
$ 
When evaluated at $q=1$, the formal characters of these two  modules are the same. It follows from the linear independence of ungraded formal characters of irreducible $R_\al$-modules that $W_1$ and $W_2$ have the same composition factors, but possibly with different degree shifts. We also know that the graded multiplicity of $L(\be,\de^{(i)})$ in $W_1=\De(\be,\de^{(i)})$ is  $1$. By Lemma~\ref{LDualInd}, we have that $W_1^{\circledast}\simeq W_2$, so  the graded multiplicity of $L(\be,\de^{(i)})$ in $W_2$ is also $1$.  
In view of Lemma~\ref{LPBW2} and Example~\ref{ExMinPair}(i), in the Grothendieck group $[\mod{R_\al}]$ we now have
$$
[W_i]=[L(\be,\de^{(i)})]+c_i[L_\rho]\qquad (i=1,2)
$$
for some graded multiplicities $c_i\in\A$ such that $\barinv c_1=c_2$. 

To compute $c_1$ and $c_2$, we look at the multiplicity of the weight $\bi^{\{n\}}$ in $W_1$. By induction, $\bi^{\{n-1\}}$ is extremal in $L_{\be}$. Let $N$ be a $\circledast$-selfdual irreducible $R_\al$-module such that 
$
N\cong \tilde f_{i}^{n+1} \tilde f^{na_{k-1}}\dots\tilde f_{i_1}^{na_1}1_F.
$
By Proposition~\ref{PProdIrrMult1}, $\bi^{\{n\}}$ is an extremal weight for $W_1$. 
An elementary computation using Proposition~\ref{PProdIrrMult1} also shows that $N$ appears in $W_1$ with graded multiplicity $q_i$. So we must have $N\simeq L_\al$, and $c_1=q_i$. We have proved (i) and (iv). Part (ii) easily follows from (i), and  (ii) implies (iii) by induction on $n$. 
\end{proof}

\subsection{\boldmath Cuspidal modules $L_{n\de-\al_i}$} \label{SS-}
Fix $i\in I'$. In this section we consider the cuspidal modules corresponding to the real roots of the form $n\de-\al_i$ for $i\in I'$. 
Recall that we have $i_k=i$ and $a_k=1$ for the extremal weight $\bi$ of $L_{\de,i}$ picked in (\ref{EExtrIm}). So in view of Corollary~\ref{CGreat}  and Lemma~\ref{LEpsLDe}, 
the weight 
$$
\bj=i_1^{a_1}\dots i_{k-1}^{a_{k-1}}
$$ 
is an extremal weight of $L_{\de-\al_i}$. We will use the notation 
$$
\bi^{[n]}:=i_1^{n}\dots i_{e-1}^{n}i_e^{n-1}\in\words_{n\de-\al_i}.
$$

\begin{Proposition} \label{PCuspSpecial-} 
Let $i\in I'$, $n\in\Z_{>1}$, and $\al=n\de-\al_i$, $\be=(n-1)\de-\al_i$. Then:
\begin{enumerate}
\item[{\rm (i)}] The standard module $\De(\de^{(i)},\be)=L_{\de,i}\circ L_\be$ has composition series of length two with head $L(\de^{(i)},\be)$ and socle $L_{\al}\lan (\al_i,\al_i)/2\ran$. 
\item[{\rm (ii)}] We have 
$$\CH L_\al=\frac{1}{q_i-q_i^{-1}}\big((\CH L_{\de,i})\circ (\CH L_\be) - (\CH L_\be)\circ (\CH L_{\de,i})\big).$$ 
\item[{\rm (iii)}] We have 
$$\CH L_\al=\frac{1}{q_i-q_i^{-1}}\sum_{m=0}^n (-1)^{n-m}(\CH L_{\de,i})^{\circ m}\circ (\CH L_{\de-\al_i}) \circ (\CH L_{\de,i})^{\circ (m)}.$$
\item[{\rm (iv)}] The weight $\bi^{[n]}$ is an extremal weight of $L_\al$. 
\end{enumerate}
\end{Proposition}
\begin{proof}
The proof is similar to that of Proposition~\ref{PCuspSpecial+}.
\end{proof}


\end{document}